\documentclass{amsart}
\usepackage{}
\usepackage{amssymb}

\usepackage{amssymb, latexsym, MnSymbol}

\usepackage{amssymb, amscd, latexsym, graphicx, psfrag}

\newcommand{\si}{\sigma}

\newcommand{\Ga}{\Gamma}

\newcommand{\bC}{\mathbb{C}}

\newcommand{\bL}{\mathbb{L}}

\newcommand{\bQ}{\mathbb{Q}}

\newcommand{\bZ}{\mathbb{Z}}
\newcommand{\bH}{\mathbb{H}}

\newcommand{\cA}{\mathcal{A}}
\newcommand{\cB}{\mathcal{B}}

\newcommand{\cD}{\mathcal{D}}
\newcommand{\cE}{\mathcal{E}}
\newcommand{\cG}{\mathcal{G}}
\newcommand{\cL}{\mathcal{L}}
\newcommand{\cI}{\mathcal{I}}
\newcommand{\cM}{\mathcal{M}}
\newcommand{\cO}{\mathcal{O}}

\newcommand{\cR}{\mathcal{R}}
\newcommand{\cQ}{\mathcal{Q}}

\newcommand{\cH}{\mathcal{H}}

\newcommand{\id}{\mathrm{id}}
\newcommand{\Ob}{\mathrm{Ob}}
\newcommand{\Hom}{\mathrm{Hom}}
\newcommand{\End}{\mathrm{End}}
\newcommand{\age}{\mathrm{age}}

\newcommand{{\inv} }{\mathrm{inv}}
\newcommand{\ev}{\mathrm{ev}}
\newcommand{\Aut}{\mathrm{Aut}}
\newcommand{\pt}{\mathrm{pt}}
\newcommand{\Res}{\mathrm{Res}}

\newcommand{\val}{ {\mathrm{val}} }
\newcommand{\vir}{{\mathrm{vir}}}
\newcommand{\CR}{  {\mathrm{CR}}  }
\newcommand{\Conj}{\mathrm{Conj}}
\newcommand{\diag}{\mathrm{diag}}
\newcommand{\Cont}{\mathrm{Cont}}
\newcommand{\Spec}{\mathrm{Spec}}

\newcommand{\one}{\mathbf{1}}

\newcommand{\be}{\mathbf{e}}

\newcommand{\bt}{\mathbf{t}}
\newcommand{\bq}{\mathbf{q}}
\newcommand{\bp}{\mathbf{p}}
\newcommand{\bw}{\mathbf{w}}
\newcommand{\bGa}{\mathbf{\Ga}}

\newcommand{\zero}{\mathbf{0}}
\newcommand{\bmu}{\boldsymbol\mu}

\newcommand{\su}{\mathsf{u}}

\newcommand{\sw}{\mathsf{w}}

\newcommand{\tS}{ {\tilde{S}} }
\newcommand{\tR}{ {\tilde{R}} }

\newcommand{\vGa}{\vec{\Ga}}

\newcommand{\Mbar}{\overline{\cM}}

\newtheorem{dummy}{dummy}[section]
\newtheorem{lemma}[dummy]{Lemma}
\newtheorem{theorem}[dummy]{Theorem}

\newtheorem{definition}[dummy]{Definition}
\newtheorem{remark}[dummy]{Remark}
\newtheorem{example}[dummy]{Example}

\begin{document}

\title{Equivariant Gromov-Witten Theory of GKM Orbifolds}

\author{Zhengyu Zong}
\address{Zhengyu Zong, Yau Mathematical Sciences Center, Tsinghua University, Jin Chun Yuan West Building,
Tsinghua University, Haidian District, Beijing 100084, China}
\email{zyzong@mail.tsinghua.edu.cn}

\begin{abstract}
In this paper, we study the all genus Gromov-Witten theory for any GKM orbifold $X$. We generalize the Givental formula which is studied in the smooth case in \cite{Giv2} \cite{Giv3} \cite{Giv4} to the orbifold case. Specifically, we recover the higher genus Gromov-Witten invariants of a GKM orbifold $X$ by its genus zero data. When $X$ is toric, the genus zero Gromov-Witten invariants of $X$ can be explicitly computed by the mirror theorem studied in \cite{CCIT} and our main theorem gives a closed formula for the all genus Gromov-Witten invariants of $X$. When $X$ is a toric Calabi-Yau 3-orbifold, our formula leads to a proof of the Remodeling Conjecture in \cite{Fang-Liu-Zong3}.

\end{abstract}

\maketitle

\tableofcontents

\section{Introduction}

\subsection{background and motivation}
Let $X$ be an algebraic GKM manifold, which means that there exists an algebraic torus $T$ acting on $X$ such that there are only finitely many fixed points and finitely many 1-dimensional orbits. In the sequence of papers \cite{Giv2} \cite{Giv3} \cite{Giv4}, Givental studies the all genus equivariant Gromov-Witten theory of a GKM manifold $X$. He obtains a formula for the full descendent potential of $X$ and conjectures that the same formula is true for any $X$ with semisimple Frobenius structure. This formula is often referred to Givental formula. It expresses the higher genus Gromov-Witten invariants of $X$ in terms of the genus 0 data. The key point is that the genus 0 data in this formula can be expressed in abstract terms of semisimple Frobenius structures of the quantum cohomology of $X$ and vice versa. So one can reconstruct the higher genus Gromov-Witten theory of $X$ by the Frobenius structure of its quantum cohomology.

When $X$ is toric, one can build the mirror symmetry between $X$ and its Landau-Ginzburg mirror. The mirror theorem (see \cite{Giv0} \cite{Giv1} and \cite{Lian-Liu-Yau1} \cite{Lian-Liu-Yau2} \cite{Lian-Liu-Yau3} ) for smooth toric varieties gives an isomorphism between the quantum cohomology ring of $X$ and the Jacobi ring of its Landau-Ginzburg mirror. So one can identify the Frobenius structures of these two rings. In particular, the quantum differential equations in A and B-model can be identified. Under this identification, the genus 0 data in Givental formula can be described explicitly on B-model side. For example, the data coming from the fundamental solution of the quantum differential equation can be given by oscillatory integrals over the Lefschetz thimbles on B-model side, the norms of the canonical basis is related to the Hessians at the critical points of the super potential and so on. Of course, the above procedure can also be generalized to other homogeneous spaces such as Grassmannians.

It is natural to ask whether the Givental formula can be generalized to the orbifold case. On the Gromov-Witten theory side, one wants to study the higher genus equivariant orbifold Gromov-Witten theory for GKM orbifolds such as toric orbifolds and other homogeneous orbifold spaces. So we want to obtain an orbifold Givental formula so that we can recover the higher genus data by the Frobenius structures of the target $X$. When $X$ is toric, the orbifold mirror theorem is proved in \cite{CCIT}. Just as in the smooth case, one can identify the quantum cohomology ring of $X$ to the Jacobi ring of its mirror. So we can give nice explicit explanations of the Frobenius structure of $X$ in terms of corresponding structures in B-model. Therefore if we can obtain an orbifold version of the Givental formula, the higher genus Gromov-Witten theory of $X$ can be recovered effectively by concrete data. This idea also applies to many other GIT quotients where similar mirror symmetry phenomenon arises \cite{Che-Cio-Kim}.

Perhaps, the most interesting case is when $X$ is a toric Calabi-Yau 3-orbifold. In this case, the \emph{Remodeling Conjecture} is established in \cite{Ma} \cite{BKMP1} and \cite{BKMP2}. The Remodeling Conjecture gives us a higher genus B-model which itself arises naturally in the matrix-model theory. The Remodeling Conjecture claims that the B-model higher genus potential can be identified with the A-model higher genus potential, which can be viewed as a all genus mirror symmetry statement. The higher genus B-model potential is obtained by applying the Eynard-Orantin recursion \cite{EO07} to the mirror curve of $X$. This potential has a graph sum formula \cite{DOSS} which has exactly the same form of the corresponding graph sum of Givental formula on A-model side . The identification of the higher genus potentials is then reduced to the identification of Frobenius structures which can be deduced from the genus 0 mirror theorem \cite{CCIT}. So the proof of the orbifold Givental formula plays a crucial role in proving the Remodeling Conjecture (see \cite{Fang-Liu-Zong3}). Besides, one needs to study the Frobenius structures that appear in the Givental formula carefully in order to match them to the corresponding structures on the B-model side. To prove the Remodeling Conjecture \cite{Fang-Liu-Zong3} is one of the main reasons for the author to study the orbifold Givental formula.

For higher dimensional toric orbifolds, one can try to generalize the Eynard-Orantin recursion to higher dimensional varieties. This may generalize the Remodeling Conjecture to higher dimensional toric orbifolds. This may be a further application of the orbifold Givental formula.

\subsection{Statement of the main result and outline of the proof}
In this paper, we prove the Givental quantization formula for the all genus equivariant Gromov-Witten potential for any GKM orbifold $X$. We consider the $T$-equivariant Chen-Ruan orbifold cohomology $H^*_{\CR,T}(X;\cQ')$ together with the $T$-equivariant orbifold product $\star_X$ and the $T$-equivariant orbifold Poincar\'{e} pairing $\langle ,\rangle_{X,T}$ on $H^*_{\CR,T}(X;\cQ')$. Here $\cQ'$ is a field extension of $\bC$ by including some $T-$equavairant parameters. The triple
$$
(H^*_{\CR,T}(X;\cQ'),\star_X,\langle ,\rangle_{X,T})
$$
is a Frobenius algebra over $\cQ'$ (see Section \ref{Frobenius}). Further more, we will explicitly construct a canonical basis $\{\bar\phi_\mu\}_{\mu\in\Sigma_X}$ for $H^*_{\CR,T}(X;\cQ')$, which means that $(H^*_{\CR,T}(X;\cQ'),\star_X,\langle ,\rangle_{X,T})$ is semisimple.

In Section \ref{Frobenius}, we generalize the classical $T$-equivariant orbifold product $\star_X$ to the $T$-equivariant \emph{quantum product} $\star_t$. Let $H:=H^*_{\CR,T}(X;\cQ'\otimes_\bC N(X))$ where $N(X)$ is the Novikov ring. Given a point $t\in H$ and any two cohomology classes $a,b\in H^*_{\CR,T}(X;\cQ'\otimes_\bC N(X))$, the quantum product $\star_t$ of $a$ and $b$ at $t$ is defined to be
$$\langle a\star_t b,c\rangle_{X,T}=\llangle a,b,c\rrangle_{0,3}^{X,T},$$
where $c\in H^*_{\CR,T}(X;\cQ'\otimes_\bC N(X))$ is any cohomology class and $\llangle a,b,c\rrangle_{0,3}^{X,T}$ is defined by the genus zero Gromov-Witten invariants of $X$. Under the quantum product $\star_t$, the formal completion
\begin{eqnarray*}
\hat{H}:=\Spec \left(\cQ'\otimes_\bC N(X)[[t^\mu]]\right)
\end{eqnarray*}
is a formal Frobenius manifold over the ring $\cQ'\otimes_\bC N(X)$ (see Section \ref{Formal completion}). It is the formal completion of $H$ at the origin. The set of global vector fields $\Gamma(\hat{H},T\hat{H})$ is a free $\cO_{\hat{H}}(\hat{H})$-module and the triple $(\Gamma(\hat{H},T\hat{H}),\star_t,\langle, \rangle_{X,T})$ is a Frobenius algebra over the ring $\cO_{\hat{H}}(\hat{H})=\cQ'\otimes_\bC N(X)[[t^\mu]]$. It is called the \emph{quantum cohomology} of $X$ and is denoted by $QH^*_{\CR,T}(X)$. We will show that the semi-simplicity of $(H^*_{\CR,T}(X;\cQ'),\star_X,\langle ,\rangle_{X,T})$ implies the semi-simplicity of $\hat{H}$. We will also define another product $\bullet$ on $\Gamma(\hat{H},T\hat{H})$ to define another Frobenius algebra $A:=(\Gamma(\hat{H},T\hat{H}),\bullet,\langle, \rangle_{X,T})$. The two Frobenius algebras $QH^*_{\CR,T}(X)$ and $A$ are isomorphic under a linear operator $\Psi$.

We will consider the Dubrovin connection $\nabla^{\hat{H}}$ on the tangent bundle $T\hat{H}$:
$$\nabla^{\hat{H}}_{\mu}=\frac{\partial}{\partial t^{\mu}}-\frac{1}{z}\tilde{\phi}_{\mu}\star_t$$
and the \emph{quantum differential equation}
$\nabla^{\hat{H}} \tau(z)=0$.

By Teleman's classification of semisimple cohomological field theories \cite{Tel}, the semi-simplicity of $\hat{H}$ implies the following orbifold Givental quantization formula:
\medskip
\paragraph{\bf Theorem \ref{ancestor}}
{\it There exists a fundamental solution $\tS=\Psi \tR(z)e^{U/z}$ to the quantum differential equation $\nabla^{\hat{H}} \tau(z)=0$ in Theorem \ref{fundamental} such that}
\begin{equation}\label{eqn:ancestor}
\cA_X(t)=\widehat{\Psi}\widehat{\tR}\cD_{I_A}.
\end{equation}
Here $U=\diag(u_1,\cdots,u_N)$ with $u_i$ the canonical coordinate of $\hat{H}$,  $\widehat{\Psi}$ is the operator $\cG(\Psi^{-1}\bq)\mapsto \cG(\bq)$ for any element $\cG$ in the Fock space, $\widehat{\tR}$ is the quantization of the $R-$matrix $\widehat{\tR}$, $\cA_X(t)$ is the ancestor potential of $X$, and $\cD_{I_A}$ is the descendent potential of the trivial cohomological field theory on $A$.

By considering the $S-$operator $S_t$ defined by
$$\langle a,S_tb\rangle_{X,T}=\langle a,b \rangle_{X,T}+\llangle a,\frac{b}{z-\psi}\rrangle_{0,2}^{X,T},$$
we obtain the Givental quantization formula for total descendent potential $\cD_X$ of $X$:
\medskip
\paragraph{\bf Theorem \ref{ancestor-descendent}}
\begin{equation}\label{eqn:ancestor-descendent}
\cD_X=\exp(F_1(t))\widehat{S_t}^{-1}\cA_X(t)=
\exp(F_1(t))\widehat{S_t}^{-1}\widehat{\Psi}\widehat{\tR}\cD_{I_A}.
\end{equation}
Here $F_1(t)$ is the genus one primary Gromov-Witten potential of $X$.

Let $F_{g,k}^{X,T}(\bt,t)$ and $\bar{F}_{g,k}^{X,T}(\bt,t)$ be the descendent potential and the ancestor potential with primary insertions respectively. Then Theorem \ref{ancestor-descendent} is equivalent to
\medskip
\paragraph{\bf Theorem \ref{descendent-ancestor}}
{\it For $2g-2+k> 0$, we have the following relation}
\begin{equation}\label{eqn:descendent-ancestor}
F_{g,k}^{X,T}(\bt,t)=\bar{F}_{g,k}^{X,T}([S_t\bt]_+,t).
\end{equation}
Here $[S_t\bt]_+$ is the part of $S_t\bt$ containing nonnegative powers of $z$.

We will also rewrite the Givental quantization formula in terms of the graph sum formula in Section \ref{graph}. The graph sum formula is crucial in the proof of the Remodeling Conjecture in \cite{Fang-Liu-Zong3}.

In the above quantization formula, the missing higher genus potential is the genus one primary Gromov-Witten potential $F^{X,T}_{1,0}(t)$ of $X$. The following theorem gives an explicit formula of $F^{X,T}_{1,0}(t)$ in terms of the Frobenius structure:
\medskip
\paragraph{\bf Theorem \ref{F1}}
\begin{equation}\label{eqn:F1}
dF^{X,T}_{1,0}(t)=\sum_{i=1}^N\frac{1}{48}d\log\Delta^i+\sum_{i=1}^N\frac{1}{2}(\tR_1)_i^{\,\ i}du^i.
\end{equation}
Here $(\Delta^i)^{-1}=\langle \frac{\partial}{\partial u^i},\frac{\partial}{\partial u^i}\rangle_{X,T}$.

Since we are working with the equivariant Gromov-Witten theory, the Frobenius manifold $\hat{H}$ is not conformal. Therefore, the ambiguity of the $R-$matrix $\tR$ cannot be fixed by the method in the conformal case by using the Euler vector field. However, by Theorem \ref{fundamental} the ambiguity is a constant matrix which allows us to fix the ambiguity by passing to the case when $t=0,Q=0$, where $Q$ is the Novikov variable. The discussion on the ambiguity of the $R-$matrix for general cohomological field theory appears in Teleman's paper \cite[Section 8.2]{Tel}. In the case of equivariant Gromov-Witten theory, the ambiguity of the $R-$matrix can be fixed by using the smooth/orbifold quantum Riemann-Roch theorem \cite{CG} \cite{Tse}. In Givental's original paper \cite{Giv2} and \cite{Giv3}, he fixes the ambiguity of the $R-$matrix for GKM manifolds by using the quantum Riemann-Roch theorem \cite{CG}. The orbifold generalization of the quantum Riemann-Roch theorem is obtained by Tseng in \cite{Tse}. In the case of $[\bC^r/G]$ for $G$ a finite abelian group, Brini-Cavaleri-Ross \cite{BCR} fix the ambiguity of the $R-$matrix by using the orbifold Riemann-Roch theorem \cite{Tse} and apply it to prove the open Crepant Transformation Conjecture for $[\bC^2/\bZ_n]\times \bC$ (See Section \ref{CTC}). In our case, we use the orbifold Riemann-Roch theorem to fix the ambiguity of the $R-$matrix for general GKM orbifolds. More concretely, we have the following reconstruction theorem from genus zero data:
\medskip
\paragraph{\bf Theorem \ref{reconstruction}}
{\it The descendent potential $F_{g,k}^{X,T}(\bt,t)$ of a GKM orbifold $X$ can be uniquely reconstructed from the operator $\tR$ which is uniquely determined by the quantum multiplication law and the property}
\begin{equation}\label{eqn:reconstruction}
(\tR_{j}^{\,\ i}) |_{t=0,Q=0}=\diag\big(((P_\sigma)_{j}^{\,\ i})\big)
\end{equation}
where $\diag\big(((P_\sigma)_{j}^{\,\ i})\big)$ is a constant matrix which is explicitly given by Theorem \ref{ambiguity}.

\subsection{Applications and future work}
The orbifold Givental quantization formula has many interesting applications. We discuss two aspects among these applications. The first aspect is about mirror symmetry. More concretely, the orbifold Givental quantization formula leads to a proof of the Remodeling Conjecture for all genus open-closed orbifold Gromov-Witten invariants of an arbitrary semi-projective toric Calabi-Yau 3-orbifold. The second aspect of applications is about the all genus Crepant Transformation Conjecture. There are two cases which are particularly interesting: the all genus open-closed Crepant Transformation Conjecture for toric Calabi-Yau 3-orbifolds and the all genus Crepant Transformation Conjecture for the equivariant total descendent potential for (possibly noncompact) toric orbifolds.

\subsubsection{The Remodeling Conjecture for semi-projective toric Calabi-Yau 3-orbifolds}
The Remodeling Conjecture was proposed by Bouchard-Klemm-Mari\~{n}o-Pasquetti \cite{BKMP1} \cite{BKMP2} based on the work of Eynard-Orantin \cite{EO07} and Mari\~{n}o \cite{Ma}. This conjecture builds a correspondence between the higher genus open-closed Gromov-Witten potentials of an arbitrary semi-projective toric Calabi-Yau 3-orbifold $X$ and the Eynard-Orantin invariants of the mirror curve of $X$. It can be viewed as an all genus open-closed mirror symmetry for semi-projective toric Calabi-Yau 3-orbifolds. The Remodeling Conjecture is proved in \cite{Fang-Liu-Zong3}. The key idea in this proof is that we
can realize the A-model and B-model higher genus potentials as
quantizations on two isomorphic semi-simple Frobenius structures. The A-model quantization is given by the orbifold Givental quantization formula studied in this paper and the B-model quantization is given by the Eynard-Orantin recursion on the mirror curve. So the Remodeling Conjecture can be reduced to the identification of the Frobenius structures on A-model and B-model. The Frobenius structures are genus zero data and hence this identification follows from the genus zero mirror theorem for toric orbifolds \cite{CCIT}.

\subsubsection{The all genus Crepant Transformation Conjecture}\label{CTC}
The Crepant Transformation Conjecture was first proposed by Ruan \cite{Ruan1, Ruan2}
and later generalized to various situations  by Bryan-Graber \cite{BG},
Coates-Corti-Iritani-Tseng \cite{CCIT}, Coates-Ruan \cite{CoRu}, etc.
The CRC relates the orbifold Gromov-Witten theory of a Gorenstein orbifold to the Gromov-Witten theory of its crepant resolution. In general, the higher genus Crepant Transformation Conjecture is difficult to formulate and prove. The orbifold Givental quantization formula is a powerful tool to prove higher genus Crepant Transformation Conjecture.

A nice result about the all genus Crepant Transformation Conjecture is obtained by Coates-Iritani \cite{CI}. They proved the all genus Crepant Transformation Conjecture for the (non-equivariant) total descendent potential for an arbitrary compact toric orbifold $X$. In this case, the Frobenius manifold coming from the (non-equivariant) quantum cohomology of $X$ is conformal. By using the orbifold Givental quantization formula for equivariant Gromov-Witten theory studied in this paper, we can generalize the result in \cite{CI} to prove the all genus Crepant Transformation Conjecture for the equivariant total descendent potential for possibly noncompact toric orbifolds. In this case, the underlying Frobenius manifold is not conformal and the ambiguity for the $R-$matrix is fixed by the method in Section 6 rather than by the Euler vector field.

Another interesting case is the all genus Crepant Transformation Conjecture for noncompact target spaces. In \cite{Zhou}, Zhou proved the the all genus Crepant Transformation Conjecture for the descendent potential for $[\bC^2/\bZ_n]$ and for the primary potential for $[\bC^2/\bZ_n]\times \bC$. Later Brini-Cavaleri-Ross \cite{BCR} formulated the all-genus open Crepant Transformation Conjecture for Hard-Lefshetz toric Calabi-Yau 3-orbifolds in terms of
winding neutral open potentials and proved this conjecture for $[\bC^2/\bZ_n]\times \bC$. After that Brini-Cavalieri \cite{BrCa} proved this conjecture for $[\bC^3/(\bZ_2\times\bZ_2)]$. The strategies in \cite{BCR} and \cite{BrCa} are enlightening and have an influence on the future works on the all-genus open Crepant Transformation Conjecture and on orbifold Gromov-Witten theory. In a forthcoming paper \cite{FLZ4}, we will use the Remodeling Conjecture proved in \cite{Fang-Liu-Zong3} to prove the all genus open Crepant Transformation Conjecture for general semi-projective toric Calabi-Yau 3-orbifolds.

\subsection{Overview of the paper}
In section 2, we first study the geometry of an GKM orbifold $X$ and then move on to the equivariant Chen-Ruan cohomology of $X$. It turns out that the equivariant Chen-Ruan cohomology ring is a semisimple Frobenius algebra.

In section 3, we first review several definitions about the equivariant Gromov-Witten theory of $X$. Then we will deduce the semisimplicity of the Frobenius manifold of $X$ from the semisimplicity of its classical equivariant Chen-Ruan cohomology. After that we will consider the quantum differential equation and its fundamental solutions. One of them is given explicitly by the 1-primary 1-descendent genus 0 potential. This solution will play a special role in the later context.

In section 4, we review the quantization procedure which will be used in the later sections.

In section 5, we will obtain the orbifold Givental formula (Theorem \ref{ancestor}, Theorem \ref{ancestor-descendent}, and Theorem \ref{descendent-ancestor}) by applying Teleman's result \cite{Tel} of classification of 2D semisimple cohomological field theories to our case. Since the Frobenius manifold of $X$ is semisimple, the cohomological field theory of $X$ coming from the Gromov-Witten theory lies in the same orbit of the trivial field theory of the same Frobenius structure under certain group actions. The Givental formula can be derived using this group action point of view. After that we give a more explicit graph sum formula (Theorem \ref{graph sum}) for the higher genus potential of $X$. This graph sum formula is crucial in the proof of the Remodeling Conjecture \cite{Fang-Liu-Zong3}. We also use this graph formula to derive a formula for the genus-one primary Gromov-Witten potential of $X$ (Theorem \ref{F1}).

In section 6, we fix the ambiguity of the $R-$operator in Givental formula. Since we are working with equivariant Gromov-Witten theory, our Frobenius structure is not conformal. Therefore the ambiguity with the $R-$operator in Givental formula cannot be fixed by the usual method using the Euler vector field. Instead, we use the structure of the solution space to the quantum differential equation (Theorem 5.1). Then we know that the ambiguity is a constant matrix. So we study the case when the degree is 0 and when there is no primary insertion and compare the Givental formula in this case with the orbifold quantum Riemann-Roch theorem in \cite{Tse}. In the end, we obtain the reconstruction theorem by expressing the ambiguity with the $R-$operator in terms of the explicit constant matrix given by orbifold quantum Riemann-Roch theorem.

\subsection{Acknowledgments}
I wish to express my deepest thanks to my advisor Chiu-Chu Melissa Liu. Her guidance has been always enlightening for me and helps me through the whole process of my work. This paper could not be possible without my advisor Chiu-Chu Melissa Liu. I also wish to thank Bohan Fang and Davesh Maulik for their helpful communications and to thank Andrea Brini, Renzo Cavalieri, and Dustin Ross for their comments on the first version of this paper. The author is partially supported by the start-up grant at Tsinghua University.

\section{Equivariant Chen-Ruan cohomology of GKM orbifolds}
In this section, we discuss the geometry and basic properties of any GKM orbifold $X$. We will study the classical equivariant Chen-Ruan cohomology ring of $X$ and construct its canonical basis.

The concept of a GKM manifold is first established in \cite{GKM} by Goresky-Kottwitz-MacPherson. An algebraic GKM manifold is a smooth algebraic variety with an algebraic action of a torus $T=(\bC^*)^m$, such that there are finitely many torus fixed points and finitely many one-dimensional orbits. Examples of algebraic GKM manifolds include toric manifolds, Grassmanians, flag manifolds and so on. The advantage of a GKM manifold $X$ is that one can study the classical equivariant cohomology of $X$ and the Atiyah-Bott localization via the combinatorics tool called GKM graphs. This might be the original motivation for people to study GKM manifolds. By generalizing the classical Atiyah-Bott localization to the virtual localization, one can compute the equivariant Gromov-Witten invariants of an algebraic GKM manifold in terms of summing over GKM graphs (see \cite{Gra-Pan1} page 20-21 for more details). The localization procedure in this case is completely similar to that in the toric case.

In this section, we study the classical equivariant Chen-Ruan cohomology of any GKM orbifold $X$, which is a generalization of the GKM manifold. The discussion of equivariant Gromov-Witten theory of $X$ is in the next section.

\subsection{GKM orbifolds}
Let $X$ be an $r-$dimensional smooth Deligne-Mumford stack with a quasi-projective coarse moduli space. Let $T=(\bC^*)^m$ be an algebraic torus acting on $X$.

\begin{definition} We say that $X$ is a \emph{GKM orbifold} if
\begin{enumerate}
\item There are finitely many $T-$fixed points.

\item There are finitely many one-dimensional orbits.

\end{enumerate}
\end{definition}
In our definition, $X$ can be noncompact and can have nontrivial generic stabilizers.

Let $p_1,\cdots, p_n$ be the $T-$fixed points of $X$. These points may be stacky and locally around each $p_\sigma$, the tangent space $T_{p_\si}X$ is isomorphic to $[\bC^r/G_\sigma]$ with $G_\sigma$ a finite group and with the $r$ axes the corresponding $r$ one dimensional orbits containing $p_\sigma,\sigma=1,\cdots,n$. Since $T$ acts on $X$, we know that the action of $T$ commutes with the action of $G_\sigma$. On the other hand, since there are only finitely many 1-dimensional orbits, the characters of the $T-$action along any two axes are linearly independent. So the image of the representation $\rho_\sigma: G_\sigma\to GL(r,\bC)$ must be contained in the maximal torus. Therefore, $\rho_\sigma$ splits into $r$ one-dimensional representations $\chi_{\sigma j}:G_\sigma\to \bC^*$, $j=1,\cdots,r$. Let $\bmu_{l_{\sigma j}}=\chi_{\sigma j}(G_\sigma)=\{z\in\bC^*|z^{l_{\sigma j}}=1\}\subseteq \bC^*$, $\sigma=1,\cdots,n, j=1,\cdots,r$.

Let $N=\Hom(\bC^*, T)$ be the lattice of 1-parameter subgroups of $T$ and $M=\Hom(T,\bC^*)$ the lattice of irreducible characters of $T$. Then $M$ is the dual lattice of $N$ and we have a canonical identification $M\cong H^2_T(\pt,\bZ)$. Let $\bw_{\sigma j}$ be the character of the $T-$action along the $j-$th tangent direction at $p_\sigma$. Then $\bw_{\sigma j}$ lies in $H^2_T(\pt,\bQ)\cong M_\bQ:=M\otimes_\bZ \bQ$ and $\bw_{\sigma i}$ and $\bw_{\sigma j}$ are linearly independent for any $i\neq j$. The rational coefficient here is used to adapt the orbifold structure at $p_\sigma$.

\subsection{Chen-Ruan orbifold cohomology}
In this subsection, we review the definition and basic properties of the Chen-Ruan orbifold cohomology \cite{Chen-Ruan04} of a smooth Deligne-Mumford stack $X$.

\subsubsection{The inertia stack}
\begin{definition}
Let $X$ be a Deligne-Mumford stack. The inertia stack $\cI X$ associated to $X$ is defined to be the fiber product
$$\cI X=X\times_{\Delta, X\times X,\Delta}X,$$
where $\Delta: X\to X\times X$ is the diagonal morphism.
\end{definition}
The objects in the category $\cI X$ can be described as
$$\Ob(\cI X)=\{(x,g)|x\in \Ob(X),g\in \Aut_{X}(x)\}.$$
The morphisms between two objects in the category $\cI X$ are
$$\Hom_{\cI X}((x_1,g_1),(x_2,g_2))=\{h\in\Hom_X(x_1,x_2)|hg_1=g_2h\}.$$
In particular,
$$\Aut_{\cI X}(x,g)=\{h\in\Aut_X(x)|hg=gh\}.$$

There is a natural projection $q:\cI X\to X$ which, on objects level, sends $(x,g)$ to $x$. There is also an involution map $\iota:\cI X\to\cI X$ which sends $(x,g)$ to $(x,g^{-1})$. The inertial stack $\cI X$ is in general not connected even if $X$ is connected. Suppose $X$ is connected and let
$$\cI X=\bigsqcup_{i\in I}X_i$$
be the disjoint union of connected components. There is a distinguished component
$$X_0=\{(x,\id_x)|x\in\Ob(X)\}$$
which is isomorphic to $X$. The restriction of the involution map $\iota$ to $X_i$ gives an isomorphism between $X_i$ and another connected component of $\cI X$. In particular, the restriction of $\iota$ to $X_0$ gives an identity map.

\begin{example}\label{BG}
Let $G$ be a finite group. Consider the quotient stack $\cB G:=[\pt/G]$ which is called the classifying space of $G$. There is only one object $x$ in $\cB G$ and $\Hom (x,x)=G$. By definition, the objects of $\cI \cB G$ are
$$\Ob(\cI \cB G)=\{(x,g)|g\in G\}.$$
The morphisms between two objects $(x_1,g_1),(x_2,g_2)$ are
$$\Hom((x_1,g_1),(x_2,g_2))=\{g\in G|gg_1=g_2g\}=\{g\in G|g_1=g^{-1}g_2g\}.$$
So we have
$$\cI \cB G\cong [G/G]$$
where $G$ acts on $G$ by conjugation. Therefore
$$\cI \cB G=\bigsqcup_{(h)\in \Conj(G)}(\cB G)_{(h)}=\bigsqcup_{(h)\in \Conj(G)}[\pt/C(h)]$$
where $(h)$ is the conjugacy class of $h\in G$ and $C(h)$ is the centralizer of $h$ in $G$.
\end{example}

\subsubsection{Chen-Ruan orbifold cohomology}
As vector spaces, the Chen-Ruan orbifold cohomology \cite{Chen-Ruan04} of a smooth Deligne-Mumford stack $X$ is the same as the usual cohomology of the inertia stack $\cI X$. The difference is the definition of the degree. Let us first discuss the definition of {\em age} which determines the degree of the Chen-Ruan orbifold cohomology.

Given an object $(x,g)\in\Ob(\cI X)$, we have a linear map $g:T_xX\to T_xX$ such that $g^l=\id$, where $l$ is the order of $g$. Let $\zeta=e^{2\pi i/l}$ and then the eigenvalues of $g$ are given by $\zeta^{c_1},\cdots,\zeta^{c_r}$, where $c_i\in\{0,\cdots,l-1\}, r=\textrm{dim}X$. Define
$$\age(x,g)=\frac{c_1+\cdots+c_r}{l}.$$
The function $\age:\cI X\to \bQ$ is constant on each connected component $X_i$ of $\cI X$ and we define $\age(X_i)$ to be $\age(x,g)$ for any object $(x,g)$ in $X_i$.

\begin{definition}
Let $X$ be a smooth Deligne-Mumford stack. The Chen-Ruan orbifold cohomology group of $X$ is defined to be
$$H^*_{\CR}(X):=\bigoplus_{a\in\bQ}H^a_{\CR}(X)$$
where
$$H^a_{\CR}(X)=\bigoplus_{i\in I}H^{a-2\age(X_i)}(X_i).$$
\end{definition}
If $X$ is proper, the orbifold Poincare pairing is defined to be
$$\langle\alpha,\beta\rangle_X=
\left\{\begin{array}{ll}\int_{X_i}\alpha\cup\iota^*_i(\beta), &X_j=\iota(X_i),\\
0, &\textrm{otherwise,} \end{array} \right.$$
where $\alpha\in H^*(X_i)$ and $\beta\in H^*(X_j)$.

The product structure for the Chen-Ruan orbifold cohomology is more subtle.
\begin{definition}\label{orbifold product}
For any $\alpha,\beta\in H^*_{\CR}(X)$, their orbifold product $\alpha\star_X\beta$ is defined as follows: For any $\gamma\in H^*_{\CR}(X)$,
$$\langle\alpha\star_X\beta,\gamma\rangle_X:=\langle \alpha,\beta,\gamma\rangle_{0,3,0}^X,$$
where the right hand side will be defined in Section \ref{orbifold GW}.

\end{definition}

\subsection{Equivariant Chen-Ruan cohomology of GKM orbifolds}\label{CRcoh}
In this subsection, we focus ourselves to the case when $X$ is a GKM orbifold. We will show that the equivariant Chen-Ruan cohomology of $X$ is semisimple by constructing its canonical basis explicitly.

\subsubsection{The simplest case: $X=[\bC^r/G]$}
Let $\Conj(G)$ denote the set of conjugacy classes in $G$. Then the inertia stack $\cI X$ can be described as
$$\cI X=\bigsqcup_{(h)\in \Conj(G)}[(\bC^r)^h/C(h)]=\bigsqcup_{(h)\in \Conj(G)}X_{(h)}$$
where $(h)$ is the conjugacy class of $h\in G$ and $C(h)$ is the centralizer of $h$ in $G$.

For any $i\in\{1,\cdots,r\}$ and any $h\in G$, define $c_i(h)\in [0,1)$ and $\age(h)$ by
\begin{equation}\label{eqn:ci-h}
e^{2\pi\sqrt{-1}c_i(h)} =\chi_i(h),
\end{equation}
\begin{equation}\label{eqn:age-h}
\age(h)= \sum_{j=1}^{r}c_j(h).
\end{equation}

As a graded vector space over $\bC$, the Chen-Ruan cohomology $H^*_\CR(X;\bC)$ of $X$ can be decomposed as
$$
H^*_\CR(X;\bC) = \bigoplus_{(h)\in \Conj(G)} H^*(X_{(h)}; \bC)[2\age(h)] =\bigoplus_{(h)\in \Conj(G)}\bC \one_{(h)},
$$
where $\deg(\one_{(h)})=2\age(h)$ which is independent of the choice of $h$ in its conjugacy class.

Let $\cR=H^*(\cB T;\bC)=\bC[\su_1,\cdots,\su_m]$, where
$\su_1,\cdots,\su_m$ are the first Chern classes of the universal line bundles over $\cB T$.
The $T$-equivariant Chen-Ruan orbifold cohomology $H^*_{\CR,T}(X;\bC)$
is an $\cR$-module. Given $h\in G$, define
$$
\be_h := \prod_{i=1}^r \sw_i^{\delta_{c_i(h),0}} \in \cR.
$$
Here, $\bw_{i}$ is the character of the $T-$action along the $i-$th direction. In particular,
$$
\be_1 = \prod_{i=1}^r \sw_i.
$$
Then the $T$-equivariant Euler class of $\zero_{(h)}:=[0/C(h)]$ in $X_{(h)} = [(\bC^r)^h/C(h)]$ is
$$
e_{T}(T_{\zero_{(h)}}X_{(h)}) = \be_h \one_{(h)} \in H^*_{\CR,T}(X_{(h)};\bC) = \cR \one_{(h)}.
$$

Let $\chi_1,\cdots,\chi_r: G\to \bC^*$ and $l_1,\cdots,l_r\in\bZ$ be defined as in the previous section. Define
$$\cR' =\bC[\su_1,\cdots,\su_m,\sw_1^{\frac{1}{2l_1}},\cdots,\sw_r^{\frac{1}{2l_r}}]$$
which is a finite extension of $\cR$. Let $\cQ$ and $\cQ'$ be the fractional fields of $\cR$, $\cR'$, respectively. We will consider the $T-$equivariant Chen-Ruan cohomology ring $H^*_{\CR,T}(X;\cQ')$.

The $T-$equivariant Poincar\'{e} pairing of $H^*_{\CR,T}(X;\cQ')$ is given by
$$
\langle \one_{(h)}, \one_{(h')}\rangle_{X,T} =\frac{1}{|C(h)|}\cdot\frac{\delta_{(h^{-1}),(h')} }{\be_h}\in \cQ'.
$$
The $T$-equivariant orbifold cup product  of $H^*_{\CR,T}(X;\cQ')$ is given by
$$
\one_{(h)} \star_{X} \one_{(h')} = \sum_{f\in(h),f'\in(h')}\frac{|C(ff')|}{|G|}\Bigl(\prod_{i=1}^r \sw_i^{c_i(h)+c_i(h')-c_i(ff')}\Bigr) \one_{(ff')}.
$$

Define
\begin{equation}\label{eqn:bar-one}
\bar{\one}_{(h)}:= \frac{\one_{(h)}}{\prod_{i=1}^r \sw_i^{c_i(h)}} \in H^*_{\CR,T}(X;\cQ').
\end{equation}
Then
$$
\langle \bar{\one}_{(h)}, \bar{\one}_{(h')}\rangle_{X,T} =
\frac{1}{|C(h)|}\cdot\frac{\delta_{(h^{-1}),(h')} }{\prod_{i=1}^r \sw_i}\in \cQ'
$$
and
$$
\bar{\one}_{(h)} \star_{X}\bar{\one}_{(h')} =\sum_{f\in(h),f'\in(h')}\frac{|C(ff')|}{|G|} \bar{\one}_{(ff')}.
$$

We now define a canonical basis for the semisimple algebra
$H^*_{\CR,T}(X;\cQ')$. Let $\{V_\alpha\}_{\alpha=1}^{|\Conj(G)|}$ be the set of irreducible representations of $G$ and let $\chi_\alpha$ be the character of $V_\alpha$. For any $\alpha$, define
$$\bar{\phi}_\alpha=\frac{\dim V_\alpha}{|G|}\sum_{(h)\in\Conj(G)}\chi_\alpha(h^{-1})\bar{\one}_{(h)}.$$
Let $\nu_\alpha=\bigl(\frac{\dim V_\alpha}{|G|}\bigr)^2,\alpha=1,\cdots,|\Conj(G)|$. Then we have
$$\langle \bar{\phi}_\alpha, \bar{\phi}_{\alpha'}\rangle_{X,T} =\delta_{\alpha,\alpha'}\frac{\nu_\alpha}{\prod_{i=1}^r \sw_i}$$
and
$$\bar{\phi}_\alpha\star_{X}\bar{\phi}_{\alpha'}=\delta_{\alpha,\alpha'}\bar{\phi}_\alpha.$$
Therefore, $\{\bar{\phi}_\alpha\}_{\alpha=1}^{|\Conj(G)|}$ is a canonical basis of the semisimple algebra
$H^*_{\CR,T}(X;\cQ')$.

\subsubsection{The general case}
In general, we apply the above construction to the local geometry around each fixed point $p_\sigma$ to get a set of cohomology classes $\{\bar{\phi}_{\sigma\alpha}\}_{\alpha=1}^{|\Conj(G_\sigma)|},\sigma=1,\cdots,n$. It is easy to show that
$$\langle \bar{\phi}_{\sigma\alpha}, \bar{\phi}_{\sigma'\alpha'}\rangle_{X,T} =\delta_{\sigma,\sigma'}\delta_{\alpha,\alpha'}\frac{\nu_{\sigma\alpha}}{\prod_{j=1}^r \sw_{\sigma j}}$$
and
$$\bar{\phi}_{\sigma\alpha}\star_{X}\bar{\phi}_{\sigma'\alpha'}=\delta_{\sigma,\sigma'}\delta_{\alpha,\alpha'}\bar{\phi}_{\sigma\alpha}.$$
Therefore, the algebra $H^*_{\CR,T}(X;\cQ')$ is still semisimple and $\{\bar{\phi}_{\sigma\alpha}|\sigma=1,\cdots,n,\alpha=1,\cdots,|\Conj(G_\sigma)|\}$ is a canonical basis of
$H^*_{\CR,T}(X;\cQ')$.

Sometimes, we also consider the normalized canonical basis $\{\tilde{\phi}_{\sigma\alpha}|\sigma=1,\cdots,n,\alpha=1,\cdots,|\Conj(G_\sigma)|\}$, where $\tilde{\phi}_{\sigma\alpha}$ is defined to be
$$\tilde{\phi}_{\sigma\alpha}=\sqrt{\frac{\prod_{j=1}^r \sw_{\sigma j}}{\nu_{\sigma\alpha}}}\bar{\phi}_{\sigma\alpha}.$$
With this definition, we have
$$\langle \tilde{\phi}_{\sigma\alpha}, \tilde{\phi}_{\sigma'\alpha'}\rangle_{X,T} =\delta_{\sigma,\sigma'}\delta_{\alpha,\alpha'}$$
i.e. $\tilde{\phi}_{\sigma\alpha}$ has unit length.

\section{Gromov-Witten theory, quantum cohomology and Frobenius manifolds}
In this section, we first recall some basic definitions of Gromov-Witten theory of a smooth Deligne-Mumford stack $X$. The foundation of Gromov-Witten theory of smooth Deligne-Mumford stacks is developed in \cite{AGV08}. The symplectic counterpart is developed in \cite{CR02}. After that, we specialized ourselves to the equivariant Gromov-Witten theory of a GKM orbifold $X$. Then we will focus on the genus zero case which determines the Frobenius structure on (the formal completion of) $H^*_{\CR,T}(X;\cQ')\otimes_\bC N(X)$ with $N(X)$ the Novikov ring. In particular, it gives us the structure of the quantum cohomology ring of $X$ and hence determines the quantum differential equation. We will discuss the semisimplicity of the Frobenius structure and the solutions to the quantum differential equation.

\subsection{Gromov-Witten theory of smooth Deligne-Mumford stacks}\label{orbifold GW}
Let $X$ be a smooth projective Deligne-Mumford stack. Let $E\subseteq H_2(X,\bZ)$ be the semigroup of effective curve classes. Define the Novikov ring $N(X)$ to be the completion of $\bC[E]$:
$$N(X)=\widehat{\bC[E]}=\{\sum_{\beta\in E}c_\beta Q^\beta|c_\beta\in\bC\}.$$
Denote by $\Mbar_{g,k}(X,\beta)$ the moduli space of degree $\beta$ stable maps to $X$ from genus $g$ curves with $k$ marked points. The difference from the smooth case is that the marked points and nodes of the domain curve are allowed to have orbifold structures with finite cyclic isotropy groups. There are $k$ orbifold line bundles $\bL_1,\cdots,\bL_k$ on $\Mbar_{g,k}(X,\beta)$. The fiber of $\bL_i$ at a point $[f:(C,x_1,\cdots,x_k)\to X]\in\Mbar_{g,k}(X,\beta)$ is the cotangent space at $x_i$. Let $X'$ be the coarse moduli space of $X$ and let $\Mbar_{g,k}(X',\beta)$ be the usual moduli space of stable maps to $X'$. We still have the corresponding line bundles $\bL'_1,\cdots,\bL'_k$ on $\Mbar_{g,k}(X',\beta)$. Let $\pi: \Mbar_{g,k}(X,\beta) \to \Mbar_{g,k}(X',\beta)$ be the natural map forgetting the orbifold structures. We define $\psi_i=\pi^*(c_1(\bL'_i)),i=1,\cdots,k$. There are evaluation maps $\ev_i:\Mbar_{g,k}(X,\beta)\to \cI X,i=1,\cdots,k$. Since the target of the evaluation maps is the inertia stack $\cI X$, the Chen-Ruan orbifold cohomology \cite{Chen-Ruan04} $H^*_{\CR}(X,\bC)$ plays the role of the state space. Let $\gamma_1,\cdots,\gamma_k\in H^*_{\CR}(X,\bC)$ and $a_1,\cdots,a_k\in\bZ_{\geq 0}$, define the orbifold Gromov-Witten invariants by the following correlator
$$\langle \tau_{a_1}\gamma_1,\cdots,\tau_{a_k}\gamma_k\rangle_{g,k,\beta}^X=
\int_{[\Mbar_{g,k}(X,\beta)]^\vir}
\prod_{i=1}^{k}((\ev_i^*\gamma_i)\psi_i^{a_i}).$$
where $[\Mbar_{g,k}(X,\beta)]^\vir$ is the virtual fundamental class as in the smooth case.

Suppose there exists an algebraic torus $T$ acting on $X$. Then the $T-$action on $X$ naturally induces a $T-$action on $\Mbar_{g,k}(X,\beta)$. Then the Gromov-Witten invariant $\langle \tau_{a_1}\gamma_1,\cdots,\tau_{a_k}\gamma_k\rangle_{g,k,\beta}^X$ can be computed via the following virtual localization formula (see \cite{Gra-Pan2}): If $\deg (\prod_{i=1}^{k}((\ev_i^*\gamma_i)\psi_i^{a_i})=\dim [\Mbar_{g,k}(X,\beta)]^\vir$,
$$
\langle \tau_{a_1}\gamma_1,\cdots,\tau_{a_k}\gamma_k\rangle_{g,k,\beta}^X
=\int_{[\Mbar_{g,k}(X,\beta)^T]^\vir}
\frac{\prod_{i=1}^{k}((\ev_i^*\gamma_i)\psi_i^{a_i})}{e_T(N^{\vir})}.
$$
Here $N^{\vir}$ is the virtual normal bundle of the fix locus $\Mbar_{g,k}(X,\beta)^T$ and the integrant is considered as in the $T-$equivariant cohomology ring.

The above virtual localization formula inspires people to define the equivariant Gromov-Witten invariants. Suppose there exists an algebraic torus $T$ acting on a possibly noncompact smooth Deligne-Mumford stack $X$. Let $\gamma_1,\cdots,\gamma_k\in H^*_{\CR,T}(X,\bC)$. Suppose that $\Mbar_{g,k}(X,\beta)^T$ is compact, then we define the $T-$equivariant Gromov-Witten invariant
$$
\langle \tau_{a_1}\gamma_1,\cdots,\tau_{a_k}\gamma_k\rangle_{g,k,\beta}^{X,T}
=\int_{[\Mbar_{g,k}(X,\beta)^T]^\vir}
\frac{\prod_{i=1}^{k}((\ev_i^*\gamma_i)\psi_i^{a_i})}{e_T(N^{\vir})}.
$$

\subsection{Gromov-Witten theory of GKM orbifolds}\label{GW-GKM}

\subsubsection{Correlators and their generating functions}

Let $X$ be a GKM orbifold. The $T-$action on $X$ naturally induces a $T-$action on $\Mbar_{g,k}(X,\beta)$. By standard virtual localization, it is easy to see that the fix locus $\Mbar_{g,k}(X,\beta)^T$ is a union of connected components labeled by decorated graphs. Each connected components is a product of spaces of the form $\Mbar_{h,n}(\cB G)$ for some finite group $G$ divided by the automorphism group. See \cite{Liu13} for this graph description for toric Deligne-Mumford stacks and \cite{Gra-Pan1} page 20-21 for smooth GKM manifolds. So $\Mbar_{g,k}(X,\beta)^T$ is compact.

For any nonnegative integer $a$, we consider the cohomology class $t_a=\sum_{\sigma,\alpha}t_a^{\sigma\alpha}\tilde{\phi}_{\sigma\alpha}\in H^*_{\CR,T}(X;\cQ')$. For convenience, we combine the two indices $\sigma,\alpha$ into a single index $\mu$ so that $\mu$ runs over the set $\Sigma_X:=\{(\sigma,\alpha)|\sigma=1,\cdots,n,\alpha=1,\cdots,|\Conj(G_\sigma)|\}$. So we can write $t_a$ as $t_a=\sum_{\mu\in\Sigma_X}t_a^{\mu}\tilde{\phi}_{\mu}\in H^*_{\CR,T}(X;\cQ')$. Define the genus $g$ correlator to be
$$\langle \bt(\psi_1),\cdots,\bt(\psi_k)\rangle_{g,k,\beta}^{X,T}=
\int_{[\Mbar_{g,k}(X,\beta)^T]^\vir}
\frac{\prod_{j=1}^{k}(\sum_{a=0}^{\infty}(\ev_j^*t_a)\psi_j^a)}{e_T(N^{\vir})}.$$
Here $N^{\vir}$ is the virtual normal bundle and $\psi_j$ is defined as in the last subsection. The insertion $\bt(\psi_j):=t_0+t_1\psi_j+t_2\psi_j^2+\cdots$ is viewed as a formal power series in $\psi_j$ with coefficients in $H^*_{\CR,T}(X;\cQ')$.

Let $t=\sum_{\mu\in\Sigma_X}t^{\mu}\tilde{\phi}_{\mu}\in H^*_{\CR,T}(X;\cQ')$. We will be interested in the following descendent potential with primary insertions:
$$F_{g,k}^{X,T}(\bt,t)=\sum_{s=0}^{\infty}\sum_{\beta\in E}\frac{Q^\beta}{s!}\langle \bt(\psi_1),\cdots,\bt(\psi_k),t,\cdots,t\rangle_{g,k+s,\beta}^{X,T}.$$
Sometimes we also denote $F_{g,k}^{X,T}(\bt,t)$ by the double bracket:
$$\llangle \bt(\psi_1),\cdots,\bt(\psi_k)\rrangle_{g,k}^{X,T}:=F_{g,k}^{X,T}(\bt,t).$$
We define the full descendent potential $\cD_X$ of $X$ to be
$$\cD_X:=\exp\big(\sum_{k\geq 0}\sum_{g\geq 0}\frac{\hbar^{g-1}}{k!}F_{g,k}^{X,T}(\bt,0)\big).$$

Let $\Mbar_{g,k}$ denote the moduli space of genus $g$ nodal curves with $k$ marked points. Consider the map $\pi:\Mbar_{g,k+s}(X,\beta)\to \Mbar_{g,k}$ which forgets the map to the target and the last $s$ marked points. Let $\bar\psi_i:=\pi^*(\psi_i)$ be the pull-backs of the classes $\psi_i, i=1,\cdots k$, from $\Mbar_{g,k}$. Then similarly we can define the ancestor potential with primary insertions to be
$$ \bar{F}_{g,k}^{X,T}(\bt,t)=\sum_{s=0}^{\infty}\sum_{\beta\in E}\frac{Q^\beta}{s!}\langle \bt(\bar{\psi}_1),\cdots,\bt(\bar\psi_k),t,\cdots,t\rangle_{g,k+s,\beta}^{X,T}.$$
We define the total ancestor potential $\cA_X(t)$ of $X$ to be
$$\cA_X(t):=\exp\big(\sum_{2g-2+k> 0}\frac{\hbar^{g-1}}{k!}\bar{F}_{g,k}^{X,T}(\bt,t)\big).$$
Note that the above sum is to make the moduli space $\Mbar_{g,k}$ be nonempty.

\subsubsection{The divisor equation}\label{divisor}
The following \emph{divisor equation} plays an important role in the study of Gromov-Witten theory:
\begin{eqnarray*}
&&\int_{[\Mbar_{g,k+1}(X,\beta)^T]^\vir}
\prod_{j=1}^{k}((\ev_j^*\gamma_j)\psi_j^{a_j})\cup\ev_{k+1}^*\gamma
=(\int_{\beta}\gamma)\int_{[\Mbar_{g,k}(X,\beta)^T]^\vir}
\prod_{j=1}^{k}((\ev_j^*\gamma_j)\psi_j^{a_j})+\\
&&\sum_{j=1}^{k}\int_{[\Mbar_{g,k}(X,\beta)^T]^\vir}
(\ev_1^*\gamma_1)\psi_1^{a_1}\cdots(\ev_{j-1}^*\gamma_{j-1})\psi_{j-1}^{a_{j-1}}
(\ev_j^*(\gamma_j\star_X\gamma))\psi_j^{a_j-1}
(\ev_{j+1}^*\gamma_{j+1})\psi_{j+1}^{a_{j+1}}\cdots(\ev_k^*\gamma_k)\psi_k^{a_k} ,
\end{eqnarray*}
where $\gamma\in H^2_{T}(X;\cQ')$ and $\gamma_1,\cdots,\gamma_k\in H^*_{\CR,T}(X;\cQ')$.
In particular, if there are only primary insertions, we only have the first term on the right hand side.

Moreover, we have the divisor equation for the ancestor invariants (see \cite{Man}):
$$
\int_{[\Mbar_{g,k+1}(X,\beta)^T]^\vir}
\prod_{j=1}^{k}((\ev_j^*\gamma_j)\bar{\psi}_j^{a_j})\cup\ev_{k+1}^*\gamma
=(\int_{\beta}\gamma)\int_{[\Mbar_{g,k}(X,\beta)^T]^\vir}
\prod_{j=1}^{k}((\ev_j^*\gamma_j)\bar{\psi}_j^{a_j}).
$$
For $t\in H^*_{\CR,T}(X;\cQ')$, let $t=t'+t''$ where $t'\in H^2_{T}(X;\cQ')$ and $t''$ is the linear combination of elements of degrees not equal to 2 and degree 2 twisted sectors. Here we view the equivariant variables in $H^*_{\CR,T}(X;\cQ')$ as degree 0 elements. Then it is easy to see that
$$ \bar{F}_{g,k}^{X,T}(\bt,t)=\sum_{s=0}^{\infty}\sum_{\beta\in E}\frac{Q^\beta e^{\int_{\beta}t'}}{s!}\langle \bt(\bar{\psi}_1),\cdots,\bt(\bar\psi_k),t'',\cdots,t''
\rangle_{g,k+s,\beta}^{X,T}.$$
Let $(t^{\mu'})$ be the coordinates of $t'$ with respect to a given basis $\{\gamma_{\mu'}\}$ in $H^2_{T}(X;\cQ')$ and let $(t^{\mu''})$ be the coordinates of $t''$ with respect to a given basis $\{\gamma_{\mu''}\}$.
Then $\bar{F}_{g,k}^{X,T}(\bt,t)$ is an element in $\cQ'[[Q^\beta\prod_{\mu'}e^{t^{\mu'}\int_{\beta}\gamma_{\mu'}},t^{\mu''}]]_{\beta\in E}[t^{\mu}_a]$. In particular, the Novikov variable $Q$ can be eliminated in this case. However, we do need the Novikov variable $Q$ when we study the descendent invariants.

\subsection{Frobenius manifolds and semisimplicity}

\subsubsection{Frobenius manifolds and reduction of the coefficient rings}\label{Frobenius}
In this section, we focus ourselves to the genus zero case. The genus zero data of $X$ determines a Frobenius structure on $H^*_{\CR,T}(X;\cQ')\otimes_\bC N(X)=H^*_{\CR,T}(X;\cQ'\otimes_\bC N(X))$ which is going to be proved to be semisimple. For a general introduction to Frobenius manifolds and quantum cohomology, the reader is referred to \cite{Man} and \cite{Lee-Pan}.

First we have the following definition of Frobenius algebra.
\begin{definition}
Let $A$ be a commutative and finitely generated free algebra over a ring $R$ with a unit $\one$. Then $A$ is called a Frobenius algebra over $R$ if it is equipped with an $R-$linear symmetric non-degenerate quadratic form $\langle,\rangle$ such that $\langle ab,c\rangle=\langle a,bc\rangle$ for any $a,b,c\in A$.

\end{definition}

\begin{example}
Recall that in Section \ref{CRcoh}, we defined the Chen-Ruan cohomology ring $H^*_{\CR,T}(X;\cQ')$, the orbifold product $\star_X$ and the $T$-equivariant Poincare pairing $\langle ,\rangle_{X,T}$. The triple $(H^*_{\CR,T}(X;\cQ'),\star_X,\langle ,\rangle_{X,T})$ is a Frobenius algebra over $\cQ'$.
\end{example}

Sometimes one may need to consider a family of Frobenius algebras, which are packaged in a nice way. This gives us the motivation to consider Frobenius manifolds.

\begin{definition}\label{Frobenius manifold}
A complex Frobenius manifold consists of four data $(M,g,A,\one)$, where
\begin{enumerate}
\item $M$ is a complex manifold of dimension $N$,
\item $g$ is a holomorphic, symmetric, non-degenerate quadratic form on the tangent bundle $TM$.
\item $A$ is a holomorphic and symmetric tensor,
$$ A: (TM)^{\otimes 3}\to \cO_M$$
with $\cO_M$ the structure sheaf of $M$,
\item $\one$ is a holomorphic vector field on $M$.

\end{enumerate}
These structures satisfy the following axioms:
\begin{enumerate}
\item The metric $g$ is flat,
\item For any point $p\in M$, there exists a neighborhood $U$ of $p$ with a basis of $g-$flat holomorphic vector fields $X_1,\cdots,X_N$ over $U$ and with a holomorphic potential function $\Phi$ on $U$ such that
    $$ A(X_i,X_j,X_k)=X_iX_jX_k(\Phi),$$
\item Let the product $\star$ on $TM$ be defined as
$$ g(X\star Y,Z)=A(X,Y,Z)$$
for any holomorphic vector fields $X,Y,Z$. Then we require that $\star$ is associative,
\item $\one$ is a $g-$flat vector field and is the unit for the product $\star$.

\end{enumerate}
\end{definition}
In Definition \ref{Frobenius manifold}, it is easy to see that for any $p\in M$, the tangent space $T_pM$ is a Frobenius algebra over $\bC$. So a Frobenius manifold can be viewed as a family of Frobenius algebras varying holomorphically.

One of the motivation for us to consider the Frobenius manifold is to study the quantum cohomology of $X$. Let $N_T(X):=\cQ'\otimes_\bC N(X)$ and let $H:=H^*_{\CR,T}(X;N_T(X))$. Given a point $t\in H$ and any two cohomology classes $a,b\in H^*_{\CR,T}(X;N_T(X))$, the \emph{quantum product} $\star_t$ of $a$ and $b$ at $t$ is defined to be
$$\langle a\star_t b,c\rangle_{X,T}=\llangle a,b,c\rrangle_{0,3}^{X,T},$$
where $c\in H^*_{\CR,T}(X;N_T(X))$ is any cohomology class. Here we view $a,b,c$ as tangent vector fields on $H$. Suppose that $\llangle a,b,c\rrangle_{0,3}^{X,T}$ converges at any point $t$ (instead of viewing $\{t^{\mu}\}$ as formal variables). Then the quantum product $\star_t$ gives us a product structure on the tangent space $T_tH$ which depends holomorphically on $t$. The quantum product is associative due to the WDVV equation.

So far, $H$ is equipped with the following structures, provided that $\llangle a,b,c\rrangle_{0,3}^{X,T}$ converges at any point $t$,
\begin{enumerate}
\item A flat metric $\langle ,\rangle_{X,T}$ which is given by the $T-$equivariant Poincar\'{e} pairing on
$H^*_{\CR,T}(X;N_T(X))$.

\item An associative commutative multiplication $\star_t$ satisfying
$\langle a\star_t b,c\rangle_{X,T}=\langle a, b\star_t c\rangle_{X,T}$, on the tangent space $T_tH$ which depends holomorphically on $t$ and the structure constant $\langle a\star_t b,c\rangle_{X,T}$ is the third derivative of the function $\llangle \rrangle_{0,3}^{X,T}$. In other words, the symmetric tensor $A$ in Definition \ref{Frobenius manifold} is given by $A(\frac{\partial}{\partial t^\mu},\frac{\partial}{\partial t^\nu},\frac{\partial}{\partial t^\rho})=\frac{\partial}{\partial t^\mu}\frac{\partial}{\partial t^\nu}\frac{\partial}{\partial t^\rho}\llangle \rrangle_{0,3}^{X,T}=\llangle \tilde{\phi}_{\mu},\tilde{\phi}_{\nu},\tilde{\phi}_{\rho} \rrangle_{0,3}^{X,T}$ and the potential function $\Phi$ is given by $\llangle \rrangle_{0,3}^{X,T}$.

\item A vector field $\one$ which is flat under the metric $\langle , \rangle_{X,T}$ and is the unit for the product structure.

\end{enumerate}
With the above structures, $H$ is almost a Frobenius manifold except that we have the following problems. The main problem is that in general we do not have the convergence of $\llangle a,b,c\rrangle_{0,3}^{X,T}$ with respect to $\{t^{\mu}\}$. Secondly, we have the Novikov variable $Q$ in $\llangle a,b,c\rrangle_{0,3}^{X,T}$ and if we view it as a fixed complex number, we also have the convergence issue for $Q$. Suppose the above two convergence conditions hold, then by letting the equivariant parameters in $\cQ'$ be some fixed complex numbers, $H$ will become a complex Frobenius manifold. In general, the way to deal with the above problems is to introduce Frobenius manifolds over any ring $R$ and make formal completion of $H$ around the origin. We will do this in Section \ref{Formal completion}.

Recall that in Section \ref{divisor}, we discussed the divisor equation in Gromov-Witten theory. By using the divisor equation, we can reduce the coefficient ring of $H$ in the following way. For any elements $a,b,c\in H^*_{\CR,T}(X;\cQ')$, we have seen in Section \ref{divisor} that the double correlator $\llangle a,b,c\rrangle_{0,3}^{X,T}$ is an element in $\cQ'[[Q^\beta\prod_{\mu'}e^{t^{\mu'}\int_{\beta}\gamma_{\mu'}},t^{\mu''}]]_{\beta\in E}$. We define a new product $a\circ_t b$ to be
$$\langle a\circ_t b,c\rangle_{X}=\llangle a,b,c\rrangle_{0,3}^{X,T}|_{Q=1},$$
where $c$ is any element in $H^*_{\CR,T}(X;\cQ')$. The right hand side $\llangle a,b,c\rrangle_{0,3}^{X,T}|_{Q=1}$ is an element in $\cQ'[[\prod_{\mu'}e^{t^{\mu'}\int_{\beta}\gamma_{\mu'}},t^{\mu''}]]_{\beta\in E}$. With this new product, the space $\cH:=H^*_{\CR,T}(X;\cQ')$ is a Frobenius manifold defined over $\cQ'$ (see Definition \ref{Frobenius manifold over R} below), provided that $\llangle a,b,c\rrangle_{0,3}^{X,T}|_{Q=1}$ converges at any point $t$. Later we will use $\cH$ to study the properties of $H$. We will also make formal completion of $\cH$ in Section \ref{Formal completion}

\subsubsection{Formal completion}\label{Formal completion}

In order to avoid the convergence issue, we will work with suitable formal completions of $H$ and $\cH$ instead of working with the whole Frobenius manifolds $H$ and $\cH$. First we give the definition of Frobenius manifolds over any ring $R$.

\begin{definition}\label{Frobenius manifold over R}
Let $R$ be a commutative $\bC$-algebra. A Frobenius manifold over $R$ is a quadruple $(M,g,A,\one)$, where
\begin{enumerate}
\item $M$ is a smooth $R-$scheme of relative dimension $N$,
\item $g$ is a $R$-linear, symmetric, non-degenerate quadratic form on the tangent bundle $TM$ over $R$.
\item $A$ is a $R$-linear symmetric tensor,
$$ A: (TM)^{\otimes 3}\to \cO_M$$
with $\cO_M$ the structure sheaf of $M$,
\item $\one$ is a vector field on $M$ over $R$.

\end{enumerate}
These structures satisfy the following axioms:
\begin{enumerate}
\item The metric $g$ is flat,
\item For any point $p\in M$, there exists a neighborhood $U$ of $p$ with a basis of $g-$flat vector fields $X_1,\cdots,X_N\in \Gamma(U,TU)$ and with a potential function $\Phi\in \Gamma(U,\cO_U)$ such that
    $$ A(X_i,X_j,X_k)=X_iX_jX_k(\Phi),$$
\item Let the product $\star$ on $TM$ be defined as
$$ g(X\star Y,Z)=A(X,Y,Z)$$
for any vector fields $X,Y,Z$. Then we require that $\star$ is associative,
\item $\one$ is a $g-$flat vector field and is the unit for the product $\star$.

\end{enumerate}
\end{definition}

Let $\hat{H}$ be the formal completion of $H$ at the origin:
\begin{eqnarray*}
\hat{H}:=\Spec \left(N_T(X)[[t^\mu]]\right).
\end{eqnarray*}
Let $\cO_{\hat{H}}$ be the structure sheaf of $\hat{H}$ and let $T\hat{H}$ be the tangent bundle of $\hat{H}$. Then $T\hat{H}$ is a free $\cO_{\hat{H}}$-module of rank $N=|\Sigma_X|$ generated by $\{\frac{\partial}{\partial t^{\mu}}\}_{\mu\in\Sigma_X}$. The formal manifold $\hat{H}$ is still equipped with the following structure:
\begin{enumerate}
\item A flat metric $\langle, \rangle_{X,T}$ on $T\hat{H}$ which is given by the $\cO_{\hat{H}}$-linear extension of the Poincar\'{e} pairing on
$H^*_{\CR,T}(X;N_T(X))$.

\item A potential function $\llangle \rrangle_{0,3}^{X,T}\in \cO_{\hat{H}}(\hat{H})$ satisfying the following properties. For any $a,b,c \in H^*_{\CR,T}(X;N_T(X))$, we view them as vector fields on $\hat{H}$. Then we have a $N_T(X)$-linear symmetric tensor
    \begin{eqnarray*}
    (T\hat{H})^{\otimes 3} &\to& \cO_{\hat{H}}\\
    a\otimes b\otimes c &\mapsto& abc(\llangle \rrangle_{0,3}^{X,T})=\llangle a,b,c \rrangle_{0,3}^{X,T}.
    \end{eqnarray*}
    The multiplication $\star_t$ defined by $\langle a\star_t b,c\rangle_{X,T}=\llangle a,b,c \rrangle_{0,3}^{X,T}$ is associative and commutative.

\item A vector field $\one$ which is flat under the metric $\langle , \rangle_{X,T}$ and is the unit for the product structure.

\end{enumerate}
With the above structures, $\hat{H}$ is a formal Frobenius manifold over the ring $N_T(X)$. It is the formal completion of $H$ at the origin. The set of global sections $\Gamma(\hat{H},T\hat{H})$ is a free $\cO_{\hat{H}}(\hat{H})$-module of rank $N$:
$$
\Gamma(\hat{H},T\hat{H})=
\bigoplus_{\mu\in\Sigma_X}\cO_{\hat{H}}(\hat{H})\frac{\partial}{\partial t^{\mu}}.
$$
Under the quantum product $\star_t$, the triple $(\Gamma(\hat{H},T\hat{H}),\star_t,\langle, \rangle_{X,T})$ is a Frobenius algebra over the ring $\cO_{\hat{H}}(\hat{H})=N_T(X)[[t^\mu]]$. The triple $(\Gamma(\hat{H},T\hat{H}),\star_t,\langle, \rangle_{X,T})$ is called the \emph{quantum cohomology} of $X$ and is denoted by $QH^*_{\CR,T}(X)$. It plays the role of the Frobenius algebra $(T_tH,\star_t,\langle, \rangle_{X,T})$ for a general point $t\in H$. We say that $\hat{H}$ is semisimple if $QH^*_{\CR,T}(X)$ is semisimple.

Similarly, we define $\hat\cH$ to be
\begin{eqnarray*}
\hat{\cH}:&=&\Spec \left(\cQ'[[\prod_{\mu'}e^{t^{\mu'}\int_{\beta}\gamma_{\mu'}}]]_{\beta\in E}[[t^{\mu''}]][t^{\mu'}]\right)
\end{eqnarray*}
Note that since $E\subseteq H_2(X,\bZ)$ is finitely generated, we only have finitely many variables in the ring $\cQ'[[\prod_{\mu'}e^{t^{\mu'}\int_{\beta}\gamma_{\mu'}}]]_{\beta\in E}[[t^{\mu''}]][t^{\mu'}]$. The manifold $\hat{\cH}$ is not a formal Frobenius manifold in the sense of \cite{Lee-Pan}. However it is easy to see that $\hat\cH$ still satisfies the properties (1) (2) (3) above with the potential function $\llangle \rrangle_{0,3}^{X,T}|_{Q=1}$ and the metric $\langle , \rangle_{X,T}$, which means that $\hat\cH$ is a Frobenius manifold defined over $\cQ'[[\prod_{\mu'}e^{t^{\mu'}\int_{\beta}\gamma_{\mu'}}]]_{\beta\in E}$. The Frobenius manifold $\hat{\cH}$ can be viewed as the formal completion of $\cH$ in the directions $\{\frac{\partial}{\partial t^{\mu''}}\}$ along the subspace $H^2_{\CR,T}(X;\cQ')$, together with adding ``admissible" functions $\{\prod_{\mu'}e^{t^{\mu'}\int_{\beta}\gamma_{\mu'}}\}_{\beta\in E}$. The functions $\prod_{\mu'}e^{t^{\mu'}\int_{\beta}\gamma_{\mu'}}$ plays the role of the Novikov variables. We say that $\hat{\cH}$ is semisimple if $(\Gamma(\hat{\cH},T\hat{\cH}),\star_t,\langle, \rangle_{X,T})$ is a semisimple Frobenius algebra over the ring $\cO_{\hat{\cH}}(\hat{\cH})$.

We will work with $\hat{H}$ and $\hat\cH$ in the later context.

\subsubsection{The canonical coordinates}

Notice that at the origin $t=0,Q=0$, the quantum product $\star_0$ is just the classical equivariant orbifold product $\star_X$ introduced in Section 2. The classical equivariant Chen-Ruan cohomology $H^*_{\CR,T}(X;\cQ')$ is semisimple as we proved in Section 2. The criterion of semi-simplicity (see Lemma 18 in \cite{Lee-Pan}) in fact implies the semi-simplicity of the quantum cohomology $QH^*_{\CR,T}(X)$. For completeness, we review the criterion of semi-simplicity. Let $I\subset N_T(X)[[t^\mu]]$ be the ideal generated by $Q,t^\mu$ and define
$$
S_n:=N_T(X)[[t^\mu]]/I^n,\quad A_n=QH^*_{\CR,T}(X)\otimes_{N_T(X)[[t^\mu]]}S_n.
$$
Then we have $A_n$ is a free $S_n-$module of rank $N$ and the ring structure $\star_t$ on $QH^*_{\CR,T}(X)$ induces a ring structure $\star_{\underline{n}}$ on $A_n$. In particular,
$$
A_1=H^*_{\CR,T}(X;\cQ'),\quad S_1=\cQ',\quad\star_{\underline{1}}=\star_X.
$$
So we know that there exists a canonical basis
$$
\{\bar\phi_\mu^{(1)}:=\bar\phi_\mu|\mu\in\Sigma_X\}
$$
for $A_1$. For $n\geq 1$, let $\{\bar\phi_\mu^{(n+1)}|\mu\in\Sigma_X\}$ be the unique canonical basis of $A_{n+1}$ which is the lift of the canonical basis $\{\bar\phi_\mu^{(n)}|\mu\in\Sigma_X\}$ of $A_{n+1}$ \cite[Lemma 16]{Lee-Pan}. Then
$$
\{\bar\phi_\mu(Q,t):=\lim \bar\phi_\mu^{(n)}|\mu\in\Sigma_X\}
$$
is a canonical basis of $QH^*_{\CR,T}(X)$ i.e. $QH^*_{\CR,T}(X)$ is semisimple.

Therefore the Frobenius manifold $\hat{H}$ is semisimple. So there exists a system of \emph{canonical} coordinates $\{u^i(t)\}_{i=1}^{N}$ on $\hat{H}$, where $N=\dim H^*_{\CR,T}(X;\cQ')=\sum_{\sigma}|\Conj(G_\sigma)|$, characterized by the property that the corresponding vector fields $\{\partial/\partial u^i\}_{i=1}^{N}$ form a canonical basis of the quantum product $\star_t$. This means that
$$
\partial/\partial u^i\star_t\partial/\partial u^j=\delta_{ij}.
$$
Similarly, the Frobenius manifold $\hat\cH$ is also semisimple and there exists a system of canonical coordinates $\{v^i(t)\}_{i=1}^{N}$ on $\hat\cH$.

The canonical coordinates $\{u^i(t)\}_{i=1}^{N}$ are defined uniquely up to reorderings, signs, and additive constants in $N_T(X)$. Similar uniqueness holds for $\{v^i(t)\}_{i=1}^{N}$ with additive constants in $\cQ'$. Notice that $v^i(t)$ is an element in $\cQ'[[\prod_{\mu'}e^{t^{\mu'}\int_{\beta}\gamma_{\mu'}},t^{\mu''}]]_{\beta\in E}[t^{\mu'}]$. We choose $\{v^i(t)\}_{i=1}^{N}$ such that the constant terms are zero. Notice that $\{v^i(Q^\beta\prod_{\mu'}e^{t^{\mu'}\int_{\beta}\gamma_{\mu'}},t^{\mu'},t^{\mu''})\}_{i=1}^{N}$ is a system of canonical coordinates of $\hat{H}$.
We choose the canonical coordinates $\{u^i(t)\}_{i=1}^{N}$ such that $u^i(t)=v^i(Q^\beta\prod_{\mu'}e^{t^{\mu'}\int_{\beta}\gamma_{\mu'}},t^{\mu'},t^{\mu''})$. This indicates that $u^i(t)$ is an element in $\cQ'[[Q^\beta\prod_{\mu'}e^{t^{\mu'}\int_{\beta}\gamma_{\mu'}},t^{\mu''}]]_{\beta\in E}[t^{\mu'}]$ and the constant term is zero.

Let $\Delta_i:=\frac{1}{\langle\partial/\partial u^i,\partial/\partial u^i \rangle_{X,T}}$. Denote by $\Psi$ the transition matrix between flat and normalized canonical basis: $\Delta_i^{-\frac{1}{2}}du^i=\sum_{\mu\in\Sigma_X}\Psi^{\,\ i}_{\mu}dt^{\mu}$. Here we use the convention that the left index of a matrix is for the rows and the right index is for the columns. The entries of $\Psi$ are elements in $\cQ'[[Q^\beta\prod_{\mu'}e^{t^{\mu'}\int_{\beta}\gamma_{\mu'}},t^{\mu''}]]_{\beta\in E}$. It is easy to see that $\Psi^{-1}=\Psi^T$ where $\Psi^T$ is the transpose of $\Psi$.

\subsection{Solutions to the quantum differential equations}\label{sol}
We consider the Dubrovin connection $\nabla^{\hat{H}}$ on the tangent bundle $T\hat{H}$:
$$\nabla^{\hat{H}}_{\mu}=\frac{\partial}{\partial t^{\mu}}-\frac{1}{z}\tilde{\phi}_{\mu}\star_t$$
for any $\mu\in\Sigma_X$. Here $z$ is a formal variable. The equation $\nabla^{\hat{H}} \tau(z)=0$ for a section $\tau(z)\in\Gamma(\hat{H}, T\hat{H})_z$ is called the \emph{quantum differential equation}. Here the section $\tau(z)$ and the space $\Gamma(\hat{H}, T\hat{H})_z$ are parametrized by the formal variable $z$. In later context, there will be different meanings for the parametrization of $z$. For example, we will consider $\tau(z)$ as power series in $1/z$ and in this case $\Gamma(\hat{H}, T\hat{H})_z$ is equal to $\Gamma(\hat{H}, T\hat{H})[[1/z]]$. On the other hand, we will also consider another type of parametrization in Theorem \ref{fundamental} in which case $\tau(z)$ is a Laurent polynomial in $z$ and $\Gamma(\hat{H}, T\hat{H})_z$ is equal to $\Gamma(\hat{H}, T\hat{H})((z))$. Therefore, the section $\tau(z)$ lies in different spaces in these cases.

In the same way, we can also define the Dubrovin connection $\nabla^{\hat{\cH}}$ on the tangent bundle $T\hat{\cH}$ to be
$$\nabla^{\hat{\cH}}_{\mu}=\frac{\partial}{\partial t^{\mu}}-\frac{1}{z}\tilde{\phi}_{\mu}\circ_t,$$
and we get the quantum differential equation $\nabla^{\hat{\cH}} \tau(z)=0$ on $\hat{\cH}$.

Let
$$
(T\hat{H})^{f,z}\subset T\hat{H}_z
$$
be the subsheaf of $\nabla^{\hat{H}}$-flat sections. Again there will be different parametrization of $z$ on $T\hat{H}_z$, which may be equal to $T\hat{H}\otimes_{\bC}\bC[[1/z]]$ or $T\hat{H}\otimes_{\bC}\bC((z))$ in later context. Then $(T\hat{H})^{f,z}$ is a sheaf of $N_T(X)\otimes_{\bC}\bC[[1/z]]$-modules or of $N_T(X)\otimes_{\bC}\bC((z))$-modules of rank $N$. When $z=\infty$, the Dubrovin connection $\nabla^{\hat{H}}$ becomes the Levi-Civita connection and $(T\hat{H})^{f,\infty}$ is the sheaf of flat vector fields.

Suppose we have a section $S\in\End(T\hat{H})=\Gamma(\hat{H},T\hat{H}\otimes T\hat{H}^\star)$. Then it defines a map
$$
S:\Gamma(\hat{H},T\hat{H})\to\Gamma(\hat{H},T\hat{H})
$$
from the free $\cO_{\hat{H}}(\hat{H})$-module $\Gamma(\hat{H},T\hat{H})$ to itself. Suppose we have a family of sections $S(z)\in\End(T\hat{H})_z$ parametrized by $z$. Again there will be different parametrization of $z$ on $S(z)$ and on $\End(T\hat{H})_z$ in later context. The section $S(z)$ may be a power series in $1/z$ or a product of operators as described in Theorem \ref{fundamental} which is a Laurent polynomial in $z$. The operator $S(z)$ is called a \emph{fundamental solution} to the quantum differential equation $\nabla^{\hat{H}} \tau(z)=0$ if the $\cO_{\hat{H}}(\hat{H})$-linear map
$$
S(z):\Gamma(\hat{H},T\hat{H})\to\Gamma(\hat{H},T\hat{H})_z
$$
restricts to a $N_T(X)$-linear injection
$$
S(z):\Gamma(\hat{H},(T\hat{H})^{f,\infty})\to\Gamma(\hat{H},(T\hat{H})^{f,z})_z
$$
between $N_T(X)$-modules and the image spans $\Gamma(\hat{H},(T\hat{H})^{f,z})_z$. The fundamental solution to the quantum differential equation $\nabla^{\hat{\cH}} \tau(z)=0$ is defined in a similar way.

Now let us describe one special fundamental solution, which connects the ancestor potential and the descendent potential of $X$ in Theorem \ref{ancestor-descendent}. Consider the operator $S_t$ defined as follows: for any $a,b\in H^*_{\CR,T}(X;N_T(X))$,
$$\langle a,S_tb\rangle_{X,T}=\langle a,b \rangle_{X,T}+\llangle a,\frac{b}{z-\psi}\rrangle_{0,2}^{X,T}.$$
The operator $S_t$ satisfies the following nice property: for any $a\in H^*_{\CR,T}(X;N_T(X))$, if we view $a$ as a flat vector field on $\hat{H}$, the section $S_t a$ satisfies the quantum differential equation i.e. we have
$$\nabla^{\hat{H}} S_t a=0.$$
This means that $S_t$ is a fundamental solution to the quantum differential equation. The proof for $S_t$ being a fundamental solution can be found in \cite{Cox-Kat} for the smooth case and in \cite{Iri} for the orbifold case which is a direct generalization of the smooth case.

The operator $S_t=\one+S_1/z+S_2/z^2+\cdots$ is a formal power series in $1/z$ with operator-valued coefficients.

\section{quantization of quadratic Hamiltonians}
In this section, we review the basic concepts of the quantization of quadratic Hamiltonians (see \cite{Giv3} for more details). The quantization procedure provides a way to recover the higher genus theory from the genus zero data which we will use in the next section.

\subsection{Symplectic space formalism}
So far, we have been working on (a formal neighborhood of) the space $H=H^*_{\CR,T}(X;N_T(X))$ which provides us the Frobenius structure and state space of the corresponding Gromov-Witten theory. When we consider the descendent theory of $X$, however, additional parameters are needed. As we have seen in section 3.1, the insertion $\bt(\psi)=t_0+t_1\psi+t_2\psi^2+\cdots$ is a formal power series in $\psi$ with an integer index that keeps track in the power of $\psi$. Similarly, the $S-$operator $S_t$ studied in the previous section is a formal power series in $1/z$. These phenomena lead to the study of the symplectic space formalism.

Let $z$ be a formal variable. We consider the space $\bH$ which is the space of Laurent polynomials in one variable $z$ with coefficients in $H$. We define the symplectic form $\Omega$ on $\bH$ by
$$\Omega(f,g)=\Res_{z=0}\langle f(-z),g(z))\rangle_{X}dz$$
for any $f,g\in\bH$. Note that we have $\Omega(f,g)=-\Omega(g,f)$. There is a natural polarization $\bH=\bH_+\oplus \bH_-$ corresponding to the decomposition $f(z,z^{-1})=f_+(z)+f_-(z^{-1})z^{-1}$ of laurent polynomials into polynomial and polar parts. It is easy to see that $\bH_+$ and $\bH_-$ are both Lagrangian subspaces of $\bH$ with respect to $\Omega$.

Introduce a Darboux coordinate system $\{p^\mu_a,q^\nu_b\}$ on $\bH$ with respect to the above polarization. This means that we write a general element $f\in\bH$ in the form
$$\sum_{a\geq 0,\mu\in\Sigma_X}p^\mu_a\tilde{\phi}^\mu(-z)^{-a-1}+\sum_{b\geq 0,\nu\in\Sigma_X}q^\nu_b\tilde{\phi}_\nu z^b,$$
where $\{\tilde{\phi}^\mu\}$ is the dual basis of $\{\tilde{\phi}_\mu\}$. Denote
\begin{eqnarray*}
\bp(z):&=&p_0(-z)^{-1}+p_1(-z)^{-2}+\cdots\\
\bq(z):&=&q_0z+q_1z^2+\cdots,
\end{eqnarray*}
where $p_a=\sum_\mu p^\mu_a \tilde{\phi}^\mu$ and $q_b=\sum_\mu q^\nu_b\tilde{\phi}_\nu$.

Recall that when we discussed the Gromov-Witten theory of $X$, we introduced the formal power series $\bt(z)=t_0+t_1z+t_2z^2+\cdots$. With $z$ replaced by $\psi$, $\bt$ appears as the insertion in the genus $g$ correlator. We relate $\bt(z)$ to the Darboux coordinates by introducing the \emph{dilaton shift}: $\bq(z)=\bt(z)-\one z$. The dilaton shift appears naturally in the quantization procedure. We will explain this phenomenon as a group action on Cohomological field theories in the next section.

\subsection{Quantization of quadratic Hamiltonians}
Let $A:\bH\to\bH$ be a linear infinitesimally symplectic transformation, i.e. $\Omega(Af,g)+\Omega(f,Ag)=0$ for any $f,g\in\bH$. Under the Darboux coordinates, the quadratic Hamiltonian
$$f\to\frac{1}{2}\Omega(Af,f)$$
is a series of homogeneous degree two monomials in $\{p^\mu_a,q^\nu_b\}$. Let $\hbar$ be a formal variable and define the quantization of quadratic monomials as
$$\widehat{q^\mu_aq_b^\nu}=\frac{q^\mu_aq_b^\nu}{\hbar},\widehat{q^\mu_ap_b^\nu}=q^\mu_a\frac{\partial}{\partial q^\nu_b},
  \widehat{p^\mu_ap_b^\nu}=\hbar \frac{\partial}{\partial q^\mu_a}\frac{\partial}{\partial q^\nu_b}.
$$
We define the quantization $\widehat{A}$ by extending the above equalities linearly. The differential operators $\widehat{q^\mu_aq_b^\nu},\widehat{q^\mu_ap_b^\nu},\widehat{p^\mu_ap_b^\nu}$ act on the so called Fock space \emph{Fock} which is the space of formal functions in $\bt(z)\in\bH_+$. For example, the descendent potential and ancestor potential are regarded as elements in \emph{Fock}. The quantization operator $\widehat{A}$ does not act on \emph{Fock} in general since it may contain infinitely many monomials. However, the actions of quantization operators in our paper are well-defined. The quantization of a symplectic transform of the form $\exp(A)$, with $A$ infinitesimally symplectic, is defined to be $\exp(\widehat{A})=\sum_{n\geq 0}\frac{\widehat{A}^n}{n!}$.

\section{Higher genus structure}
In this section, we recover the higher genus data of a GKM orbifold $X$ by its Frobenius structure. Recall that in section 3.2, we have proved that the Frobenius manifolds $\hat{H}$ and $\hat\cH$ of any GKM orbifold $X$ are semisimple. This means that the underlining cohomological field theory, which comes from the Gromov-Witten theory of $X$, is semisimple. So we can use Teleman's result \cite{Tel}, which classifies the 2D semisimple field theories, to express the ancestor potential of $X$ in terms of certain group action (which is basically the action of quantization operators) on the trivial cohomological field theory with the same Frobenius structure. The key point here is that the cohomological field theory that we are considering is not conformal since we are working with the equivariant Gromov-Witten theory. So the ambiguity with the corresponding group element, which acts on the trivial cohomological field theory, cannot be fixed by the usual method using Euler vector field. Instead, we will fix the ambiguity by studying the degree 0 case of the ancestor potential and by using the structure of the solution space to the quantum differential equation. We will study this ambiguity in the next section. In the end, the descendent potential is related to the ancestor potential by the $S-$operator $S_t$ in the standard way (see \cite{Giv3} and \cite{Coa}).

\subsection{The quantization procedure and group actions on cohomological field theories}
Recall that in section 3.2, we defined the transition matrix $\Psi$ between flat and normalized canonical basis: $\Delta_i^{-\frac{1}{2}}du^i=\sum_{\mu\in\Sigma_X}\Psi^{\,\ i}_{\mu}dt^{\mu}$. If we view $\Psi$ as an operator on $H^*_{\CR,T}(X;N_T(X))$ which sends the flat basis $\{\tilde{\phi}_\mu\}$ to the normalized canonical basis $\{\Delta_i^{\frac{1}{2}}\frac{\partial}{\partial u^i}\}$, then $(\Psi^{\,\ i}_{\mu})$ is the corresponding matrix expression under the basis $\{\tilde{\phi}_\mu\}$. Note that when $t=0,Q=0$, the normalized canonical basis $\{\Delta_i^{\frac{1}{2}}\frac{\partial}{\partial u^i}\}$ coincides with the flat basis $\{\tilde{\phi}_\mu\}$. So we have a canonical 1-1 correspondence between the two sets $\{\Delta_i^{\frac{1}{2}}\frac{\partial}{\partial u^i}\}$ and $\{\tilde{\phi}_\mu\}$. Under this correspondence, we can use the same index ($i$ or $\mu$) for both of the two sets of basis. Sometimes we also denote $\Delta_i^{\frac{1}{2}}\frac{\partial}{\partial u^i}$ by $\tilde{\phi}_i(t,Q)$. Therefore we have
$$\tilde{\phi}_i(0,0)= \tilde{\phi}_i.$$
Let $U$ denote the diagonal matrix $\diag(u^1,\cdots,u^N)$. Using the above correspondence, we can define the operator $e^{U/z}$ which has the matrix expression $e^{U/z}$ under the basis $\{\tilde{\phi}_\mu\}$.

Now we can state the following theorem which characterizes the solution space of the quantum differential equation. Recall that there is a formal variable $z$ in the quantum differential equation, which can give different parametrization on the fundamental solutions. \emph{In the following theorem, we consider the fundamental solutions with entries in $\cQ'((z))[[Q^\beta\prod_{\mu'}e^{t^{\mu'}\int_{\beta}\gamma_{\mu'}},t^{\mu''}]]_{\beta\in E}[t^{\mu'}]$}. The proof of this theorem can be found in \cite{D} \cite{G97} \cite{Giv2}.

\begin{theorem}\label{fundamental}
\begin{enumerate}
\item The quantum differential equation $\nabla^{\hat{H}} \tau(z)=0$ has a fundamental solution in the form: $\tilde{S}=\Psi \tR(z)e^{U/z}$ with entries in $\cQ'((z))[[Q^\beta\prod_{\mu'}e^{t^{\mu'}\int_{\beta}\gamma_{\mu'}},t^{\mu''}]]_{\beta\in E}[t^{\mu'}]$, where $\tR(z)=\one+\tR_1z+\tR_2z^2+\cdots$ is a formal matrix power series satisfying the unitary condition $\tR^*(-z)\tR(z)=\one$, where $\tR^*$ is the adjoint of $\tR$.\\

\item The series $\tR(z)$ satisfying the unitary condition in (1) is unique up to right multiplication by diagonal matrices $\exp(a_1z+a_3z^3+a_5z^5+\cdots)$ where $a_{2k-1}$ are constant diagonal matrices.\\

\item In the case of conformal Frobenius manifolds the series $\tR(z)$ satisfying the unitary condition in (a) is uniquely determined by the homogeneity condition $(z\partial_z+\sum u^i\partial_{u^i})\tR(z)=0$.
\end{enumerate}
\end{theorem}

From this theorem, we know that there are ambiguities with the $R$-matrix in the fundamental solution $\tilde{S}$ if we work with non-conformal Frobenius manifolds. However, the ambiguity is a \emph{constant} matrix and so we can fix it by studying the case when $t=0,Q=0$. This will be done in the next section.

\begin{remark}
The fundamental solution $\tilde{S}$ in Theorem \ref{fundamental} is viewed as a matrix with entries in $\cQ'((z))[[Q^\beta\prod_{\mu'}e^{t^{\mu'}\int_{\beta}\gamma_{\mu'}},t^{\mu''}]]_{\beta\in E}[t^{\mu'}]$, where $\{t^{\mu'}\}$ is a coordinate system of $H^2_T(X,\bC)$. Since we choose the canonical coordinates $\{u^i(t)\}_{i=1}^{N}$ such that they vanish when $Q=0,t=0$, if we fix the powers of $Q$ and $t^\mu,\mu\in\Sigma_X$, only finitely many terms in the expansion of $e^{U/z}$ contribute. So the multiplication $\Psi \tR(z)e^{U/z}$ is well defined and the result matrix indeed has entries in $\cQ'((z))[[Q^\beta\prod_{\mu'}e^{t^{\mu'}\int_{\beta}\gamma_{\mu'}},t^{\mu''}]]_{\beta\in E}[t^{\mu'}]$.
\end{remark}

\begin{remark}[the divisor equation] One should also notice that in the quantum differential equation $\nabla^{\hat{H}} \tau(z)=0$, we can apply the divisor equation (see \cite{Man}) to the quantum product. Then we will get the corresponding quantum differential equation $\nabla^{\hat{\cH}} \tau(z)=0$ for the Frobenius manifold $\hat\cH$. So we can in fact eliminate the Novikov variable $Q$ and get a version of Theorem \ref{fundamental} for the Frobenius manifold $\hat\cH$. In this case, the solution $\tilde{S}$ is viewed as a matrix with entries in $\cQ'((z))[[\prod_{\mu'}e^{t^{\mu'}\int_{\beta}\gamma_{\mu'}},t^{\mu''}]]_{\beta\in E}[t^{\mu'}]$, where $\{t^{\mu'}\}$ is a coordinate system of $H^2_T(X,\bC)$ and $\{t^\mu\}$ is a coordinate system of $H^*_{\CR,T}(X,\bC)$. So the Novikov variable $Q$ in Theorem \ref{fundamental} just rescales $\prod_{\mu'}e^{t^{\mu'}\int_{\beta}\gamma_{\mu'}}$ and in particular we can see that the constant matrices $a_{2k-1}$ in (2) of Theorem \ref{fundamental} are also independent of $Q$.

By the divisor equation, the Novikov variable $Q$ is also redundant when we consider the ancestor potential $\cA_X(t)$ and its quantization formula (see Theorem \ref{ancestor}). However the Novikov variable $Q$ is necessary in the definition of the descendent potential in which case the divisor equation has descendent correction terms.

\end{remark}

The operator $\tR$ in Theorem \ref{fundamental} plays a central role in the quantization procedure.

Recall that in Section \ref{Formal completion}, we defined the quantum cohomology $QH^*_{\CR,T}(X)=(\Gamma(\hat{H},T\hat{H}),\star_t,\langle, \rangle_{X,T})$, which is Frobenius algebra over the ring $\cO_{\hat{H}}(\hat{H})=N_T(X)[[t^\mu]]$. Also recall that we have a natural 1-1 correspondence between $\{\Delta_i^{\frac{1}{2}}\frac{\partial}{\partial u^i}\}$ and $\{\tilde{\phi}_\mu=\frac{\partial}{\partial t^\mu}\}$ by identifying them at the origin $t=0,Q=0$. So we can use the same index $i$ for both of the two basis. Now we define another multiplication $\bullet$ on $\Gamma(\hat{H},T\hat{H})$ by
$$
\frac{\partial}{\partial t^i}\bullet \frac{\partial}{\partial t^j}=\delta_{i,j}\Delta_i^{\frac{1}{2}}\frac{\partial}{\partial t^i}.
$$
Let $A$ be the Frobenius algebra $(\Gamma(\hat{H},T\hat{H}),\bullet,\langle, \rangle_{X,T})$. Then $\{\frac{\partial}{\partial t^\mu}\}$ is a normalized canonical basis of $A$. The linear operator $\Psi$ gives an isomorphism
\begin{eqnarray*}
\Psi: A&\to& QH^*_{\CR,T}(X)\\
\frac{\partial}{\partial t^i}&\mapsto& \Delta_i^{\frac{1}{2}}\frac{\partial}{\partial u^i}
\end{eqnarray*}
between Frobenius algebras.

Before we move on to the quantization process, let us consider the potential functions of the trivial cohomological field theory $I_{A}$. Define the correlator $\langle \rangle_{g,k}^{I_{A}}$ to be
$$\langle \tau_{a_1}(\tilde{\phi}_{i_1}),\cdots,\tau_{a_k}(\tilde{\phi}_{i_k})\rangle_{g,k}^{I_{A}}=
\left\{\begin{array}{ll}\Delta_i^{g-1+k/2}\int_{\Mbar_{g,k}}\psi_1^{a_1}\cdots\psi_k^{a_k}, &\textrm{if}\quad i_1=i_2=\cdots=i_k=i,\\
0, &\textrm{otherwise}\end{array} \right.$$
where $a_1,\cdots,a_k$ are nonnegative integers. Let
$$\cD_{I_{A}}=\exp\big(\sum_{g\geq 0}\sum_{k\geq 0}\sum_{a_1,\cdots,a_k\geq 0}\sum_{i_1,\cdots,i_k\in\{1,\cdots,N\}}
\frac{\hbar^{g-1}t_{a_1}^{i_1}\cdots t_{a_k}^{i_k}}{a_1!\cdots a_k!}
\langle \tau_{a_1}(\tilde{\phi}_{i_1}),\cdots,\tau_{a_k}(\tilde{\phi}_{i_k})\rangle_{g,k}^{I_{A}}\big).$$

The following theorem is the result of the semisimplicity of the Frobenius manifold $\hat{H}$ and Teleman's classification of semisimple cohomological field theories (see \cite{Tel}):

\begin{theorem}[Givental formula for ancestor potentials of GKM orbifolds]\label{ancestor}
There exists a fundamental solution $\tS=\Psi \tR(z)e^{U/z}$ to the quantum differential equation $\nabla^{\hat{H}} \tau(z)=0$ in Theorem \ref{fundamental} such that
$$\cA_X(t)=\widehat{\Psi}\widehat{\tR}\cD_{I_A}.$$
Here $\widehat{\Psi}$ is the operator $\cG(\Psi^{-1}\bq)\mapsto \cG(\bq)$ for any element $\cG$ in the Fock space.
\end{theorem}

\begin{remark}
In \cite{Tel}, Teleman in fact considers the fundamental solution to the quantum differential equation $\nabla^{\hat{\cH}} \tau(z)=0$ for the Frobenius manifold $\hat{\cH}$. The corresponding $R-$matrix is $\tR|_{Q=1}$. Then the corresponding quantization formula is
\begin{equation}\label{quantization}
\cA_X(t)|_{Q=1}=\widehat{\Psi}|_{Q=1}\widehat{\tR}|_{Q=1}\cD_{I_A}.
\end{equation}
Theorem \ref{ancestor} is obtained by rescaling each factor $\prod_{\mu'}e^{t^{\mu'}\int_{\beta}\gamma_{\mu'}}$ in (\ref{quantization}) by $Q^\beta$.
\end{remark}

\begin{remark}\label{group}
In \cite{Tel}, $\widehat{\tR}$ is explained as an element in a certain group acting on the cohomological field theories and Theorem \ref{ancestor} is equivalent to say that $\cA_X(t)$ and $\cD_{I_{A}}$ lie in the same orbit under this group action. The dilaton shift in the quantization process is equivalent to the conjugate action of the translation $T_z$ in \cite{Tel}. The underlining Frobenius algebras of the two cohomological field theories $\cA_X(t)$ and $\cD_{I_{A}}$ are $QH^*_{\CR,T}(X)$ and $A$ respectively. And $\Psi$ gives an isomorphism between $A$ and $QH^*_{\CR,T}(X)$.

\end{remark}

Recall that we have defined a particular fundamental solution $S_t$ in Section \ref{sol} by the 1-primary, 1-descendent correlator. The quantization of the operator $S_t$ relates the ancestor potential $\cA_X(t)$ to the descendent potential $\cD_X$. Specifically, we have the following theorem:

\begin{theorem}[Givental formula for descendent potentials of GKM orbifolds]\label{ancestor-descendent}
Let $F_1(t)=\sum_{k\geq 0}\frac{1}{k!}F_{1,k}^{X,T}(\bt,0)|_{t_0=1,t_1=t_2=\cdots=0}$. Then we have
$$\cD_X=\exp(F_1(t))\widehat{S_t}^{-1}\cA_X(t)=\exp(F_1(t))\widehat{S_t}^{-1}\widehat{\Psi}\widehat{\tR}\cD_{I_A}.$$

\end{theorem}

The proof of the relation $\cD_X=\exp(F_1(t))\widehat{S_t}^{-1}\cA_t$ can be found in Theorem 1.5.1 of \cite{Coa} for the smooth case. The strategy is to use the comparison lemma type argument to relate $\psi$ to $\bar\psi$. The proof for the orbifold case is completely similar to the smooth case.

We should notice that although $\cA_X(t)$ depends on $t$, the total descendent potential $\cD_X$ is independent of $t$. For our purpose, we are more interested in the descendent potential with primary insertions i.e the function $F_{g,k}^{X,T}(\bt,t)$ defined in Section \ref{GW-GKM}. This potential function will eventually correspond to the B-model potential under the all genus mirror symmetry studied in \cite{Fang-Liu-Zong3}. The relation between $F_{g,k}^{X,T}(\bt,t)$ and the ancestor potential $\bar{F}_{g,k}^{X,T}(\bt,t)$ is even easier:

\begin{theorem}\label{descendent-ancestor}
For $2g-2+k> 0$, we have the following relation
$$F_{g,k}^{X,T}(\bt,t)=\bar{F}_{g,k}^{X,T}([S_t\bt]_+,t).$$
Here we consider $\bt=\bt(z)$ as element in $\bH$ and $[S_t\bt]_+$ is the part of $S_t\bt$ containing nonnegative powers of $z$.

\end{theorem}

The proof of the theorem can also be found in the proof of Theorem 1.5.1 of \cite{Coa}.

\subsection{The graph sum formula}\label{graph}
In this section, we describe a graph sum formula of $\bar{F}_{g,k}^{X,T}(\bt,t)$ which is equivalent to Theorem \ref{ancestor}. By using Theorem \ref{descendent-ancestor}, we obtain the graph sum formula for $F_{g,k}^{X,T}(\bt,t)$ for $2g-2+k>0$.
The graph sum formula gives us a more explicit expression of the Givental formula. In \cite{Fang-Liu-Zong3}, we prove the mirror symmetry by expressing both the A-model and B-model potentials as graph sums and by identifying each term in the graph sum. The reader is referred to \cite{D-Sha-spi} for a proof of the equivalence of the graph sum formula and the quantization formula.

In this subsection, every matrix expression of the corresponding linear operator is under the basis $\{\tilde{\phi}_\mu\}$.

Given a connected graph $\Ga$, we introduce the following notation.
\begin{enumerate}
\item $V(\Ga)$ is the set of vertices in $\Ga$.
\item $E(\Ga)$ is the set of edges in $\Ga$.
\item $H(\Ga)$ is the set of half edges in $\Gamma$.
\item $L^o(\Ga)$ is the set of ordinary leaves in $\Ga$.
\item $L^1(\Ga)$ is the set of dilaton leaves in $\Ga$.
\end{enumerate}

With the above notation, we introduce the following labels:
\begin{enumerate}
\item (genus) $g: V(\Ga)\to \bZ_{\geq 0}$.
\item (marking) $i: V(\Ga) \to \{1,\cdots,N\}$. This induces
$i:L(\Ga)=L^o(\Ga)\cup L^1(\Ga)\to \{1,\cdots,N\}$, as follows:
if $l\in L(\Ga)$ is a leaf attached to a vertex $v\in V(\Ga)$,
define $i(l)=i(v)$.
\item (height) $a: H(\Ga)\to \bZ_{\geq 0}$.
\end{enumerate}

Given an edge $e$, let $h_1(e),h_2(e)$ be the two half edges associated to $e$. The order of the two half edges does not affect the graph sum formula in this paper. Given a vertex $v\in V(\Ga)$, let $H(v)$ denote the set of half edges
emanating from $v$. The valency of the vertex $v$ is equal to the cardinality of the set $H(v)$: $\val(v)=|H(v)|$.
A labeled graph $\vGa=(\Ga,g,i,a)$ is {\em stable} if
$$
2g(v)-2 + \val(v) >0
$$
for all $v\in V(\Ga)$.

Let $\bGa(X)$ denote the set of all stable labeled graphs
$\vGa=(\Gamma,g,i,a)$. The genus of a stable labeled graph
$\vGa$ is defined to be
$$
g(\vGa):= \sum_{v\in V(\Ga)}g(v)  + |E(\Ga)|- |V(\Ga)|  +1
=\sum_{v\in V(\Ga)} (g(v)-1) + (\sum_{e\in E(\Gamma)} 1) +1.
$$
Define
$$
\bGa_{g,k}(X)=\{ \vGa=(\Gamma,g,i,a)\in \bGa(X): g(\vGa)=g, |L^o(\Ga)|=k\}.
$$
Let $\bt(z)=\sum_\mu\bt^\mu(z)\tilde{\phi}_\mu$.

We assign weights to leaves, edges, and vertices of a labeled graph $\vGa\in \bGa(X)$ as follows.

\begin{enumerate}
\item {\em Ordinary leaves.} To each ordinary leaf $l \in L^o(\Ga)$ with  $i(l)= i\in \{1,\cdots,N\}$
and  $a(l)= a\in \bZ_{\geq 0}$, we assign:
$$
(\cL^{\bt})^i_a(l) = [z^a] (\sum_{\mu,j=1,\cdots,N}\bt^\mu(z)\Psi_{\mu}^{\,\ j}
\tR_j^{\,\ i}(-z) ).
$$
\item {\em Dilaton leaves.} To each dilaton leaf $l \in L^1(\Ga)$ with $i(l)=i\in \{1,\cdots,N\}$
and $2\leq a(l)=a \in \bZ_{\geq 0}$, we assign
$$
(\cL^1)^i_a(l) = [z^{a-1}](-\sum_{j=1,\cdots,N}\frac{1}{\sqrt{\Delta^j}} \tR_j^{\,\ i}(-z)).
$$

\item {\em Edges.} To an edge connected a vertex marked by $i\in \{1,\cdots,N\}$ to a vertex
marked by $j\in \{1,\cdots,N\}$ and with heights $a$ and $b$ at the corresponding half-edges, we assign
$$
\cE^{i,j}_{a,b}(e) = [z^a w^b]
\Bigl(\frac{1}{z+w} (\delta_{i,j}-\sum_{p=1,\cdots,N}
\tR_p^{\,\ i}(-z) \tR_p^{\,\ j}(-w)\Bigr).
$$
\item {\em Vertices.} To a vertex $v$ with genus $g(v)=g\in \bZ_{\geq 0}$ and with
marking $i(v)=i$, with $k_1$ ordinary
leaves and half-edges attached to it with heights $a_1, ..., a_n \in \bZ_{\geq 0}$ and $k_2$ more
dilaton leaves with heights $a_{k_1+1}, \ldots, a_{k_1+k_2}\in \bZ_{\geq 0}$, we assign
$$
 \int_{\Mbar_{g,k_1+k_2}}\psi_1^{a_1} \cdots \psi_{k_1+k_2}^{a_{n+m}}.
$$
\end{enumerate}

We define the weight of a labeled graph $\vGa\in \bGa(X)$ to be
\begin{eqnarray*}
w(\vGa) &=& \prod_{v\in V(\Ga)} (\sqrt{\Delta^{i(v)}})^{2g(v)-2+\val(v)} \langle \prod_{h\in H(v)} \tau_{a(h)}\rangle_{g(v)}
\prod_{e\in E(\Ga)} \cE^{i(v_1(e)),i(v_2(e))}_{a(h_1(e)),a(h_2(e))}(e)\\
&& \cdot \prod_{l\in L^o(\Ga)}(\cL^{\bt})^{i(l)}_{a(l)}(l)
\prod_{l\in L^1(\Ga)}(\cL^1)^{i(l)}_{a(l)}(l).
\end{eqnarray*}
Then
$$
\log(\cA_X(t)) =
\sum_{\vGa\in \bGa(X)}  \frac{ \hbar^{g(\vGa)-1} w(\vGa)}{|\Aut(\vGa)|}
= \sum_{g\geq 0}\hbar^{g-1} \sum_{k\geq 0}\sum_{\vGa\in \bGa_{g,k}(X)}\frac{w(\vGa)}{|\Aut(\vGa)|}.
$$

By Proposition \ref{descendent-ancestor}, $F_{g,k}^{X,T}(\bt,t)$ can be obtained by $\bar{F}_{g,k}^{X,T}(\bt,t)$ by change of variables defined by the operator $S_t(z)$. So in order to get the graph sum formula for $F_{g,k}^{X,T}(\bt,t)$, we only need to modify the ordinary leaves:

\begin{enumerate}
\item[(1)'] {\em Ordinary leaves.}
To each ordinary leaf $l \in L^o(\Ga)$ with  $i(l)= i\in \{1,\cdots,N\}$
and  $a(l)= a\in \bZ_{\geq 0}$, we assign:
$$
(\mathring{\cL}^{\bt})^i_a(l) = [z^a] (\sum_{\mu,\nu,j=1,\cdots,N}(\bt^\mu(z)S_t(z)^\nu_{\,\ \mu})_+\Psi_{\nu}^{\,\ j}
\tR_j^{\,\ i}(-z) ).
$$

\end{enumerate}
We define a new weight of a labeled graph $\vGa\in \bGa(X)$ to be
\begin{eqnarray*}
\mathring{w}(\vGa) &=& \prod_{v\in V(\Ga)} (\sqrt{\Delta^{i(v)}})^{2g(v)-2+\val(v)} \langle \prod_{h\in H(v)} \tau_{a(h)}\rangle_{g(v)}
\prod_{e\in E(\Ga)} \cE^{i(v_1(e)),i(v_2(e))}_{a(h_1(e)),a(h_2(e))}(e)\\
&& \cdot \prod_{l\in L^o(\Ga)}(\mathring{\cL}^{\bt})^{i(l)}_{a(l)}(l)
\prod_{l\in L^1(\Ga)}(\cL^1)^{i(l)}_{a(l)}(l).
\end{eqnarray*}
Then by Theorem \ref{ancestor} and Proposition \ref{descendent-ancestor}, we have the following graph sum formula for $F^{X,T}_{g,k}(\bt,t)$,

\begin{theorem}\label{graph sum}
For $2g-2+k>0$,
$$
\frac{1}{k!}F^{X,T}_{g,k}(\bt,t) =
\sum_{\vGa\in \bGa_{g,k}(X)}\frac{\mathring{w}(\vGa)}{|\Aut(\vGa)|}.
$$
\end{theorem}

\subsection{A formula for $F^{X,T}_{1,0}(t)$}
One should notice that the graph sum formula in the last section is for the case when $2g-2+k>0$. The missing higher genus potential is the case when $g=1,k=0$. In this section, we derive a formula for the genus one GW potential $F^{X,T}_{1,0}(t)$. Concretely speaking, we will prove the following Theorem \ref{F1}. When $X$ is smooth, this theorem appears in \cite{G97} \cite{Giv3}.

\begin{theorem}\label{F1}
The following formula holds for the genus one Gromov-Witten potential $F^{X,T}_{1,0}(t)$:
$$
dF^{X,T}_{1,0}(t)=\sum_{i=1}^N\frac{1}{48}d\log\Delta^i+\sum_{i=1}^N\frac{1}{2}(\tR_1)_i^{\,\ i}du^i.
$$

\end{theorem}
\begin{proof}
We will use the graph sum formula for $F^{X,T}_{1,1}(\bt,t)$ to compute $F^{X,T}_{1,0}(t)$. By definition,
$$dF^{X,T}_{1,0}(t)=\sum_{\mu}\frac{\partial F^{X,T}_{1,0}}{\partial t^\mu}dt^\mu=\sum_j\frac{\partial F^{X,T}_{1,0}}{\partial u^j}du^j.$$
We have
\begin{eqnarray*}
\frac{\partial F^{X,T}_{1,0}}{\partial u^j}&=&\llangle \tilde{\phi}_j(t,Q) \rrangle^{X,T}_{1,1}.\\
\end{eqnarray*}
By Theorem \ref{graph sum}, $\llangle \tilde{\phi}_i(t,Q) \rrangle^{X,T}_{1,1}$ can be expresses as a graph sum. It is easy to see that there are only three graphs which have nonzero weights in this graph sum:
\begin{enumerate}
\item $\Gamma_1$: There is one genus one vertex $v$ and one ordinary leaf $l^0$ at $v$ with $a(l^0)=1$ in $\Gamma_1$.

\item $\Gamma_2$: There is one genus one vertex $v$, one ordinary leaf $l^0$ with $a(l^0)=0$ and one dilaton leaf $l^1$ with $a(l^1)=2$ in $\Gamma_2$.

\item $\Gamma_3$: There is one genus zero vertex $v$, one ordinary leaf $l^0$ with $a(l^0)=0$ and one edge $\cE$ which is a loop at $v$ with $a_1(\cE)=a_2(\cE)=0$ in $\Gamma_3$.

\end{enumerate}
By Theorem \ref{graph sum}, the contributions from the above three graphs are
\begin{eqnarray*}
\Cont(\Gamma_1)&=&
-\sum_{i}\sqrt{\Delta^i}\cdot\frac{1}{24}\cdot\frac{1}{\sqrt{\Delta^j}}(\tR_1)_j^{\,\ i}\\
\Cont(\Gamma_2)&=&\Delta^j\cdot\frac{1}{24}\cdot\frac{1}{\sqrt{\Delta^j}}
\sum_p\frac{1}{\sqrt{\Delta^p}}(\tR_1)_p^{\,\ j}\\
\Cont(\Gamma_3)&=&\frac{1}{2}(\tR_1)_j^{\,\ j}.
\end{eqnarray*}
Here we used the fact that $\int_{\Mbar_{1,2}}\psi_1^2=\int_{\Mbar_{1,1}\psi_1}=\frac{1}{24}$.
Therefore
\begin{eqnarray*}
dF^{X,T}_{1,0}(t)&=&\sum_j\frac{\partial F^{X,T}_{1,0}}{\partial u^j}du^j\\
&=&\sum_j\llangle \tilde{\phi}_j(t,Q) \rrangle^{X,T}_{1,1}du^j\\
&=&\sum_j(\Cont(\Gamma_1)+\Cont(\Gamma_2)+\Cont(\Gamma_3))du^j\\
&=&\frac{1}{24}\sum_{i,j}\frac{\sqrt{\Delta^j}}{\sqrt{\Delta^i}}(\tR_1)_i^{\,\ j}(du^j-du^i)+\frac{1}{2}\sum_j(\tR_1)_j^{\,\ j}du^j.
\end{eqnarray*}
Comparing with the right hand side of the equality in the statement of the theorem, we only need to show that
$$\frac{1}{24}\sum_{i,j}\frac{\sqrt{\Delta^j}}{\sqrt{\Delta^i}}(\tR_1)_i^{\,\ j}(du^j-du^i)=\sum_{i=1}\frac{1}{48}d\log\Delta^i.$$
From the quantum differential equation $\nabla\tS=0$, we get a recursive relation
$$(\Psi^{-1}d\Psi)_j^{\,\ i}=(\tR_1)_j^{\,\ i}(du^j-du^i).$$
So we have
$$\frac{1}{24}\sum_{i,j}\frac{\sqrt{\Delta^j}}{\sqrt{\Delta^i}}(\tR_1)_i^{\,\ j}(du^j-du^i)
=-\frac{1}{24}\sum_{i,j}\frac{\sqrt{\Delta^i}}{\sqrt{\Delta^j}}(\Psi^{-1}d\Psi)_j^{\,\ i}.$$
By Lemma \ref{Delta} below, we have
$$\frac{1}{\sqrt{\Delta^j}}=\sum_{\beta\in\Sigma_X}\frac{1}{\sqrt{\Delta^\beta}(0)}
\Psi_\beta^{\,\ j},$$
where $\sqrt{\Delta^\beta}(0)=\sqrt{\Delta^\beta}|_{Q=0,t=0}$.
Therefore
\begin{eqnarray*}
-\sum_{i,j}\frac{\sqrt{\Delta^i}}{\sqrt{\Delta^j}}(\Psi^{-1}d\Psi)_j^{\,\ i}&=&
-\sum_{i,j}\sqrt{\Delta^i}(\Psi^{-1}d\Psi)_j^{\,\ i}(\sum_{\beta}\frac{1}{\sqrt{\Delta^\beta}(0)}\Psi_\beta^{\,\ j})\\
&=&-\sum_{i,j}\sqrt{\Delta^i}\big(\sum_\alpha(\Psi^{-1})_j^{\,\ \alpha}
d\Psi_\alpha^{\,\ i}\big)(\sum_{\beta}\frac{1}{\sqrt{\Delta^\beta}(0)}\Psi_\beta^{\,\ j})\\
&=&-\sum_{i}\sqrt{\Delta^i}\big(\sum_{\alpha}\sum_\beta(\frac{1}{\sqrt{\Delta^\beta}(0)}
\delta_{\alpha\beta})d\Psi_\alpha^{\,\ i}\big)\\
&=&-\sum_{i}\sqrt{\Delta^i}\sum_{\alpha}\frac{1}{\sqrt{\Delta^\alpha}(0)}d\Psi_\alpha^{\,\ i}\\
&=&-\sum_{i}\sqrt{\Delta^i}d(\frac{1}{\sqrt{\Delta^i}})\\
&=&\frac{1}{2}\sum_{i}\frac{d\Delta^i}{\Delta^i}\\
&=&\frac{1}{2}\sum_{i}d\log\Delta^i.
\end{eqnarray*}
Therefore,
$$\frac{1}{24}\sum_{i,j}\frac{\sqrt{\Delta^j}}{\sqrt{\Delta^i}}(\tR_1)_i^{\,\ j}(du^j-du^i)
=-\frac{1}{24}\sum_{i,j}\frac{\sqrt{\Delta^i}}{\sqrt{\Delta^j}}(\Psi^{-1}d\Psi)_j^{\,\ i}=\frac{1}{48}\sum_{i}d\log\Delta^i$$
which finishes the proof.

\end{proof}

In the proof of Theorem \ref{F1}, the following lemma is needed:

\begin{lemma}\label{Delta}
For any $j=1,\cdots,N$,
$$\frac{1}{\sqrt{\Delta^j}}=\sum_{\beta\in\Sigma_X}\frac{1}{\sqrt{\Delta^\beta}(0)}
\Psi_\beta^{\,\ j},$$
where $\sqrt{\Delta^\beta}(0)=\sqrt{\Delta^\beta}|_{Q=0,t=0}$.
\end{lemma}
\begin{proof}
First we should notice that since $\{\frac{\tilde{\phi}_j(t,Q)}{\sqrt{\Delta^j}}\}$ is the canonical basis for the quantum product $\star_t$, multiplication by the element $\sum_j\frac{\tilde{\phi}_j(t,Q)}{\sqrt{\Delta^j}}$ is identity. On the other hand, $\one$ is the identity element in the quantum cohomology ring of $X$. So we have $\sum_j\frac{\tilde{\phi}_j(t,Q)}{\sqrt{\Delta^j}}=\one$. Similarly, since $\{\frac{\tilde{\phi}_\beta}{\sqrt{\Delta^\beta}(0)}\}$ is the canonical basis for the classical orbifold product $\star_X$ and $\one$ is the identity element in the orbifold cohomology ring of $X$, we have $\sum_\beta\frac{\tilde{\phi}_\beta}{\sqrt{\Delta^\beta}(0)}=\one$.

Note that $\langle \tilde{\phi}_i(t,Q),\tilde{\phi}_j(t,Q)\rangle_X=\delta_{ij}$. We have
$$\langle\tilde{\phi}_j(t,Q),\one\rangle_X=
\langle\tilde{\phi}_j(t,Q),\sum_i\frac{\tilde{\phi}_i(t,Q)}{\sqrt{\Delta^i}}\rangle_X
=\frac{1}{\sqrt{\Delta^j}}.
$$

On the other hand
\begin{eqnarray*}
\langle\tilde{\phi}_j(t,Q),\one\rangle_X&=&
\langle\tilde{\phi}_j(t,Q),\sum_\beta\frac{\tilde{\phi}_\beta}{\sqrt{\Delta^\beta}(0)}\rangle_X\\
&=&\langle\tilde{\phi}_j(t,Q),\sum_\beta\frac{1}{\sqrt{\Delta^\beta}(0)}\sum_i\Psi_\beta^{\,\ i}\tilde{\phi}_i(t,Q)\rangle_X\\
&=&\sum_\beta\frac{1}{\sqrt{\Delta^\beta}(0)}\sum_i\Psi_\beta^{\,\ i}\delta_{ij}\\
&=&\sum_\beta\frac{1}{\sqrt{\Delta^\beta}(0)}\Psi_\beta^{\,\ j}.
\end{eqnarray*}
Therefore
$$\frac{1}{\sqrt{\Delta^j}}=\sum_{\beta\in\Sigma_X}\frac{1}{\sqrt{\Delta^\beta}(0)}
\Psi_\beta^{\,\ j}.$$
\end{proof}

\section{Reconstruction from genus zero data}
As mentioned earlier, the operator $\tilde{R}$ is not uniquely determined since we are working with non-conformal Frobenius manifold. However, by Theorem \ref{fundamental} the ambiguity is a constant matrix which allows us to fix the ambiguity by passing to the case when $t=0,Q=0$. In this case, the domain curve is contracted to one of the torus fixed points $p_1,\cdots,p_n$ of $X$ and there is no primary insertions. So we can reduce the problem to the case when $X=[\bC^r/G]$. In this case, we can study the Gromov-Witten theory of $X$ by orbifold quantum Riemann-Roch theorem in \cite{Tse}

\subsection{The case $X=[\bC^r/G]$ and orbifold quantum Riemann-Roch theorem}
In this section, we apply the orbifold quantum Riemann-Roch theorem to $X=[\bC^r/G]$ to get a formula for $\cD_X$. Then we can compare this formula with the Givental formula in the previous section to fix the ambiguity with the operator $\tR$.

Recall that the Bernoulli polynomials $B_m(x)$ are defined by
$$\frac{te^{tx}}{e^t-1}=\sum_{m\geq 0}\frac{B_m(x)t^m}{m!}.$$
The Bernoulli numbers are given by $B_m:=B_m(0)$.

Let $X=[\bC^r/G]$. Then
$$\cI X=\bigsqcup_{(h)\in\Conj(G)} [(\bC^r)^h/C(h)]$$
and
$$H^*(\cI X;\bC)=\bigoplus_{(h)\in\Conj(G)}\bC\one_{(h)}.$$
We recall the Chen-Ruan cohomology $H^*_{\CR,T}( X;\cQ')$ studied in Section \ref{CRcoh}. For each irreducible representation $V_\alpha$ of $G$, let
$$\tilde{\phi}_\alpha=\sqrt{\prod_{j}\sw_j}\sum_{(h)\in\Conj(G)}\chi_\alpha(h^{-1})
\bar{\one}_{(h)}.$$
Then by Section \ref{CRcoh}, we have
$$
\tilde{\phi}_\alpha\star_X\tilde{\phi}_{\alpha'}
=\delta_{\alpha,\alpha'}\frac{|G|\sqrt{\prod_{j}\sw_j}}{\dim V_\alpha}\tilde{\phi}_\alpha
$$
and
$$
\langle \tilde{\phi}_\alpha,\tilde{\phi}_{\alpha'}\rangle_{X,T}=\delta_{\alpha,\alpha'}.
$$
So $\{\tilde{\phi}_\alpha\}$ is a normalized canonical basis for $H^*_{\CR,T}( X;\cQ')$.

For each integer $m\geq 0$ and $i=1,\cdots,r$, define an linear operator $A_m^i:H^*(\cI X)\to H^*(\cI X)$ by
$$A_m^i(\one_{(h)}):=B_m (c_i(h))\one_{(h)}.$$
Then we define the symplectic operator $P(z)$ to be
$$P(z):=\prod_{i=1}^{r}\exp\Bigl(\sum_{m\geq 1}\frac{(-1)^{m}}{m(m+1)} A^i_{m+1} (\frac{z}{\sw_i})^m \Bigr),$$
where $\sw_i$ is the torus character in the $i-$th tangent direction.

Let $A_0:=(H^*_{\CR,T}( X;\cQ'),\star_X,\langle,\rangle_{X,T})$ and we define the cohomological field theory $I_{A_0}$ as follows. Define the genus $g$ correlator $\langle \rangle_{g,k}^{I_{A_0}}$ to be
$$\langle \tau_{a_1}(\tilde{\phi}_{i_1}),\cdots,\tau_{a_k}(\tilde{\phi}_{i_k})\rangle_{g,k}^{I_{A_0}}=
\left\{\begin{array}{ll}(\frac{|G|\sqrt{\prod_{j}\sw_j}}{\dim V_i})^{2g-2+k}\int_{\Mbar_{g,k}}\psi_1^{a_1}\cdots\psi_k^{a_k}, &\textrm{if}\quad i_1=i_2=\cdots=i_k=i,\\
0, &\textrm{otherwise}\end{array} \right.$$
where $a_1,\cdots,a_k$ are nonnegative integers. Let
$$\cD_{I_{A_0}}=\exp\big(\sum_{g\geq 0}\sum_{k\geq 0}\sum_{a_1,\cdots,a_k\geq 0}\sum_{i_1,\cdots,i_k\in\{1,\cdots,|\Conj(G)|\}}\frac{\hbar^{g-1}t_{a_1}^{i_1}\cdots t_{a_k}^{i_k}}{a_1!\cdots a_k!}
\langle \tau_{a_1}(\tilde{\phi}_{i_1}),\cdots,\tau_{a_k}(\tilde{\phi}_{i_k})\rangle_{g,k}^{I_{A_0}}\big).$$

Then we have the following orbifold quantum Riemann-Roch theorem \cite{Tse}:
\begin{theorem}[orbifold quantum Riemann-Roch theorem for $X$]\label{qrr}
$$\cD_X=\widehat{P}\cD_{I_{A_0}}.$$
\end{theorem}



Now we make the following key observation. When $t=0,Q=0$, $\Delta_i=\frac{|G|^2\prod_{j}\sw_j}{(\dim V_i)^2}$. This means that the Frobenius algebra $QH^*_{\CR,T}(X)|_{t=0,Q=0}$ is isomorphic to $A_0$ and the isomorphism is given by
$$\tilde{\phi_i}\mapsto \tilde{\phi}_i$$
for $i=1,\cdots,|\Conj(G)|$. On the other hand, when $t=0,Q=0$ we have $\cD_X=\cA_X$ and $\Psi$ is trivial. So by comparing Theorem \ref{ancestor} and Theorem \ref{qrr}, we know that
$$\tR_{j}^{\,\ i}|_{t=0,Q=0}=P_{j}^{\,\ i},$$
where $(P_{j}^{\,\ i})$ is the matrix expression of $P$ under the basis $\{\tilde\phi_\alpha\}$. Therefore if we let $\alpha_i,\alpha_j$ be the corresponding irreducible representation, we have
\begin{equation}\label{eqn:R0}
\tR_{j}^{\,\ i}|_{t=0,Q=0}=\frac{1}{|G|}\sum_{(h)\in \Conj(G)} \chi_{\alpha_j}(h)
\chi_{\alpha_i}(h^{-1})\prod_{k=1}^r \exp\Bigl(\sum_{m=1}^\infty \frac{(-1)^m}{m(m+1)} B_{m+1}(c_k(h))(\frac{z}{\sw_k})^m \Bigr).
\end{equation}

\subsection{The general case}
In general, when $t=0,Q=0$, the domain curve is contracted to one of the torus fixed points $p_1,\cdots,p_n$. So we apply the quantum Riemann-Roch theorem to each $[\bC^r/G_\sigma]$ for $\sigma=1,\cdots,n$ and we obtain $n$ operators $P_1,\cdots,P_n$. So we have the following theorem which fixes the ambiguity with $\tR$.

\begin{theorem}\label{ambiguity}
The operator $\tR$ in theorem \ref{ancestor} is uniquely determined by the property that
$$(\tR_{j}^{\,\ i}) |_{t=0,Q=0}=\diag\big(((P_\sigma)_{j}^{\,\ i})\big)$$
where the block $((P_\sigma)_{j}^{\,\ i})$ is given by
$$(P_\sigma)_{j}^{\,\ i}=\frac{1}{|G_\sigma|}\sum_{(h)\in \Conj(G_\sigma)} \chi_{\alpha_j}(h)
\chi_{\alpha_i}(h^{-1})\prod_{k=1}^r \exp\Bigl(\sum_{m=1}^\infty \frac{(-1)^m}{m(m+1)} B_{m+1}(c_{\sigma k}(h))(\frac{z}{\sw_{\sigma k}})^m \Bigr).$$

\end{theorem}

Theorem \ref{fundamental} is proved by substituting $\Psi \tR(z)e^{U/z}$ into the quantum differential equation and solve each $(\tR)_k$ inductively (see \cite{Giv2}). In this process, the constant terms in the diagonal entries of each $(\tR)_{2k-1}$ are ambiguous and this ambiguity is fixed by Theorem \ref{ambiguity}. So the matrix $\tR(z)$ can be uniquely reconstructed from the quantum differential equation and Theorem \ref{ambiguity}. Recall that the quantum differential equation is determined by the quantum product $\star_t$ which is given by the genus zero three points function $\llangle \cdots\rrangle_{0,3}^{X,T}$. So by combining Theorem \ref{ancestor} and Theorem \ref{descendent-ancestor}, we have the following reconstruction theorem from genus zero data:


\begin{theorem}[Reconstruction from the Frobenius structure]\label{reconstruction}
The descendent potential $F_{g,k}^{X,T}(\bt,t)$ of a GKM orbifold $X$ can be uniquely reconstructed from the operator $\tR$ which is uniquely determined by the quantum multiplication law and the property
$$(\tR_{j}^{\,\ i}) |_{t=0,Q=0}=\diag\big(((P_\sigma)_{j}^{\,\ i})\big)$$
where $\diag\big(((P_\sigma)_{j}^{\,\ i})\big)$ is a constant matrix which is explicitly given by Theorem \ref{ambiguity}.

\end{theorem}

\end{document}